\DeclareSymbolFont{rsfs}{U}{rsfs}{m}{n}

\documentclass[10pt,a4paper,oneside,leqno,noamsfonts]{amsart}
           \makeatletter
           \newcommand{\mylabel}[2]{#2\def\@currentlabel{#2}\label{#1}}
           \renewcommand\@biblabel[1]{#1.}
           \makeatother
\usepackage[english]{babel}
\usepackage[dvipsnames]{xcolor}

\usepackage[noadjust]{cite}

\usepackage{graphicx,pifont}
\usepackage[a4paper,top=2.8cm,bottom=2.8cm,left=2.7cm,
           right=2.7cm,bindingoffset=5mm,marginparwidth=60pt]{geometry}
\usepackage[utf8]{inputenc}
\usepackage{braket,caption,comment,mathtools,stmaryrd}
\usepackage[usestackEOL]{stackengine}
\usepackage{multirow,booktabs,microtype,relsize}
\usepackage[colorlinks,bookmarks]{hyperref} 
      \hypersetup{colorlinks,%
            citecolor=britishracinggreen,%
            filecolor=black,%
            linkcolor=cobalt,%
            urlcolor=black}
      \numberwithin{equation}{section}

     \usepackage[capitalise]{cleveref} 
        \usepackage[baseline]{euflag}
\usepackage{mathrsfs}

\usepackage[colorinlistoftodos]{todonotes}

\definecolor{antiquewhite}{rgb}{0.98, 0.92, 0.84}
\definecolor{buff}{rgb}{0.94, 0.86, 0.51}
\definecolor{palecopper}{rgb}{0.85, 0.54, 0.4}
\definecolor{fluorescentyellow}{rgb}{0.8, 1.0, 0.0}

\definecolor{britishracinggreen}{rgb}{0.0, 0.26, 0.15}
\definecolor{cobalt}{rgb}{0.0, 0.28, 0.67}
\DeclareSymbolFont{usualmathcal}{OMS}{cmsy}{m}{n}
\DeclareSymbolFontAlphabet{\mathcal}{usualmathcal}

\newcommand{\TT}{\mathbf{T}}

\newcommand{\BA}{{\mathbb{A}}}

\newcommand{\BC}{{\mathbb{C}}}

\newcommand{\BE}{{\mathbb{E}}}

\newcommand{\BL}{{\mathbb{L}}}

\newcommand{\BP}{{\mathbb{P}}}
\newcommand{\BQ}{{\mathbb{Q}}}

\newcommand{\BZ}{{\mathbb{Z}}}

\newcommand{\CA}{{\mathcal A}}

\newcommand{\CC}{{\mathcal C}}

\newcommand{\CE}{{\mathcal E}}
\newcommand{\CF}{{\mathcal F}}

\newcommand{\CI}{{\mathcal I}}

\newcommand{\CM}{{\mathcal M}}
\newcommand{\CN}{{\mathcal N}}

\newcommand{\CQ}{{\mathcal Q}}

\newcommand{\CV}{{\mathcal V}}
\newcommand{\CW}{{\mathcal W}}

\newcommand{\CZ}{{\mathcal Z}}

\newcommand{\simto}{\,\widetilde{\to}\,}

\DeclareMathOperator{\Cone}{\mathrm{Cone}}

\newcommand{\pt}{{\mathsf{pt}}}
\newcommand{\ch}{{\mathrm{ch}}}

\DeclareMathOperator{\Hilb}{Hilb}

\DeclareMathOperator{\Quot}{Quot}

\DeclareMathOperator{\Sym}{Sym}

\DeclareMathOperator{\coh}{coh}

\DeclareMathOperator{\Coh}{Coh}

\DeclareMathOperator{\vir}{\mathrm{vir}}

\DeclareMathOperator{\Exp}{Exp}

\DeclareMathOperator{\GL}{GL}

\newcommand{\derived}{\mathbf{D}}

\newcommand*{\defeq}{\mathrel{\vcenter{\baselineskip0.5ex \lineskiplimit0pt
                     \hbox{\scriptsize.}\hbox{\scriptsize.}}}%
                     =}

\usepackage{bbm}

\newcommand{\nc}{\ensuremath{\rm nc}}
\newcommand{\fr}{\ensuremath{\rm fr}}

\newcommand\LQ{\ensuremath{\mathsf{L}}}

\usepackage{subcaption}
\usepackage[outline]{contour}
\usetikzlibrary{cd,arrows,decorations.markings,positioning,shapes,calc,patterns,tqft,intersections,shapes.misc}
\tikzset{>=Latex}
\tikzset{GaugeNode/.style={circle,draw,inner sep=0pt,minimum size=10mm}}
\tikzset{FrameNode/.style={rectangle,draw,inner sep=0pt,minimum size=9mm}}
\tikzset{token/.style={circle,double,draw=black!70,fill=black!50,inner sep=0pt,minimum size=3mm}}

\newcommand{\into}{\hookrightarrow}
\newcommand{\onto}{\twoheadrightarrow}

\DeclareFontFamily{OT1}{rsfs}{}
\DeclareFontShape{OT1}{rsfs}{n}{it}{<-> rsfs10}{}
\DeclareMathAlphabet{\curly}{OT1}{rsfs}{n}{it}

\newcommand\Hom{\operatorname{Hom}}
\newcommand\End{\operatorname{End}}

\newcommand\id{\operatorname{id}}

\newcommand{\OO}{\mathscr O}

\DeclareMathOperator{\bn}{{\bold{n}}}



\usepackage[all]{xy}
\usepackage{tikz}
\usepackage{tikz-cd}
\usepackage{adjustbox}
\usepackage{rotating}
\usepackage{comment}

\usetikzlibrary{matrix,shapes,intersections,arrows,decorations.pathmorphing}
\tikzset{commutative diagrams/arrow style=math font}
\tikzset{commutative diagrams/.cd,
mysymbol/.style={start anchor=center,end anchor=center,draw=none}}

\tikzset{
shift up/.style={
to path={([yshift=#1]\tikztostart.east) -- ([yshift=#1]\tikztotarget.west) \tikztonodes}
}
}

\theoremstyle{definition}

\newtheorem*{lemma*}{Lemma}
\newtheorem*{theorem*}{Theorem}
\newtheorem*{example*}{Example}
\newtheorem*{fact*}{Fact}
\newtheorem*{notation*}{Notation}
\newtheorem*{definition*}{Definition}
\newtheorem*{prop*}{Proposition}
\newtheorem*{remark*}{Remark}
\newtheorem*{corollary*}{Corollary}

\newtheorem*{conventions*}{Conventions}

\newtheorem{definition}{Definition}[section]

\newtheorem{example}[definition]{Example}

\newtheorem{remark}[definition]{Remark}

\newtheoremstyle{thm} 
        {3mm}
        {3mm}
        {\slshape}
        {0mm}
        {\bfseries}
        {.}
        {1mm}
        {}
\theoremstyle{thm}
\newtheorem{theorem}[definition]{Theorem}
\newtheorem{corollary}[definition]{Corollary}

\newtheorem{prop}[definition]{Proposition}

\newtheoremstyle{ex} 
        {3mm}
        {3mm}
        {}
        {0mm}
        {\scshape}
        {.}
        {1mm}
        {}
\theoremstyle{ex}

\newtheoremstyle{sol} 
        {3mm}
        {3mm}
        {}
        {0mm}
        {\scshape}
        {.}
        {1mm}
        {}
\theoremstyle{sol}

\usepackage{amsmath,float}
\usetikzlibrary{decorations.markings}

\usepackage{amssymb}
\usepackage{enumitem}
\usepackage{comment}

\usepackage{bm}

   \DeclareMathOperator{\oO}{\mathcal{O}}
   
   \DeclareMathOperator{\tf}{\mathfrak{t}}
\DeclareMathOperator{\vdim}{vdim}

\allowdisplaybreaks
\title[Gauge origami on broken lines]{Gauge origami on broken lines}

\author{Sergej Monavari}
\address{Dipartimento di Matematica “Tullio Levi-Civita”, Università degli Studi di Padova, Via Trieste 63, 35121 Padova, Italy}
\email{sergej.monavari@math.unipd.it}

\begin{document}

\maketitle
\begin{abstract}
In analogy to Nekrasov's theory of gauge origami on intersecting branes, we introduce the gauge origami moduli space  on broken lines. We realize this moduli space as a Quot scheme parametrising zero-dimensional quotients of a torsion sheaf on two intersecting affine lines, and describe it as a moduli space of quiver representations. We construct a virtual fundamental class and virtual structure sheaf, by which we define  $K$-theoretic invariants. We compute its associated partition function for all ranks, and show that it reproduces the generating series of equivariant $\chi_{y}$-genus when the moduli space is smooth. Finally,  we relate our partition function with the virtual invariants of the Quot schemes of the affine plane and Nekrasov's partition function.
\end{abstract}

\section{Introduction}
\subsection{Overview} In a series of seminal papers \cite{Nek_BPS1, Nek_BPS2, Nek_BPS3, Nek_BPS4} on the BPS/CFT correspondence, Nekrasov introduced the \emph{gauge origami partition function}, which can be realised in Type IIB string theory studying the  dynamics of  D1-branes probing  a configuration of intersecting D5-branes.
\smallbreak
In analogy to Nekrasov's theory, we introduce the gauge origami moduli space on \emph{broken lines} exploiting Grothendieck's Quot schemes.
Our construction exploits the novel techniques introduced in \cite{FM_tetra}, and we expect it to  provide a suitable  mathematical framework to properly describe the \emph{coupled vortex systems} recently proposed in string theory in  \cite{KN_Gauge1}. In fact, as it was explained in \cite{KN_Gauge1},  these  models should capture  the dynamics of D1-branes probing intersecting D3-branes in Type IIB string theory, but a complete and systematic study of their geometry and of their associated partition function is currently lacking in the literature.
\smallbreak
The main results of this paper are:
\begin{itemize}
    \item the construction of  the gauge origami moduli space $\CM_{\overline{r}, n}$ along with a natural virtual fundamental cycle, by which we define virtual invariants in equivariant $K$-theory,  see~\Cref{sec: quiver},
    \item  an explicit closed formula for the associated partition function $\CZ_{\overline{r}}(q)$ as a plethystic exponential, for all multi-ranks $\overline{r}=(r_1, r_2)$,  see~\Cref{sec:framing},
    \item a relation between  $\CZ_{\overline{r}}(q)$ and the  Nekrasov's partition function with fundamental and anti-fundamental matter,  see~\Cref{sec: ADHM}.
\end{itemize}
We explain in the next sections  our results in more details.
\subsection{The moduli space}
 Let $\overline{r}=(r_1, r_2)$ and $n\geq 0$. We set\footnote{Throughout the paper, we assume the base field to be $\BC$.} 
 \[\mathcal{C}=Z(x_1x_2) \subset \BA^2\]
 to be  the union of the two coordinate axes of $\BA^1_i=Z(x_i)\subset \BA^2$. Let $\iota_i: \BA^1_i\hookrightarrow \mathcal{C}$ be the inclusions of the irreducible components, and define the torsion sheaf
 \[\CE_{\overline{r}}=\iota_{1,*}\oO^{{r_1}}_{\BA^1_1}\oplus \iota_{2,*}\oO^{{r_2}}_{\BA^1_2}\]
on $\mathcal{C}$. We define the \emph{gauge origami moduli space   on broken lines} as Grothendieck's Quot scheme
\[\CM_{\overline{r}, n}=\Quot_\mathcal{C}(\CE_{\overline{r}},n),\]
which parametrizes flat families of zero-dimensional quotients $[\CE_{\overline{r}}\onto Q]$ on $\mathcal{C}$ of length $n$, where two such quotients are identified whenever their kernels coincide.
\smallbreak
If $\overline{r}=(0,r)$, then $\CM_{\overline{r}, n}=\Quot_{\BA^1}(\oO^r,n)$ recovers the usual Quot scheme of zero-dimensional quotients on $\BA^1$, which is a smooth irreducible variety of dimension $rn$ \cite{MR_lissite}. However, for an arbitrary choice of ranks $\overline{r}$,  the moduli space $\CM_{\overline{r}, n} $  can feature several irreducible components, possibly of different dimension. See \Cref{example: blow up} for the case with $r_1, r_2\geq 1$ and  $n=1$, where we completely work out the geometry of  $\CM_{\overline{r},1}$, which is shown have three irreducible component. In this example, two of the components are identified with the smooth Quot schemes $\Quot_{\BA_i^1}(\oO^{r_i},1)$, while the third component -- which intersects both the other two -- is isomorphic to the projective space $\BP^{r_1+r_2-1}$ and should be thought as a \emph{bubble}, parametrising  the possible degeneration of a quotient $[\oO^{r_i}_{\BA^1_i}\to Q]$ when the support of $Q$ moves towards the origin $0\in \BA^1_i$.
\subsubsection{Quiver model}
As we already discussed, $\CM_{\overline{r}, n}$ is in general not equidimensional  -- and, in particular, singular -- which makes the task of extracting relevant invariants challenging. 

To bypass this problem, we show that the moduli space $\CM_{\overline{r}, n} $ admits a description as a moduli space of quiver representations, by embedding into the \emph{non-commutative Quot scheme} $\mathcal M^{\nc}_{\overline r,n}$, which is a \emph{smooth} moduli space parametrising representations of the quiver in \Cref{fig:tetrahedron-quiver}. Our first result is that $ \CM_{\overline{r}, n} $ is globally cut out inside $\CM^{\nc}_{\overline r,n}$ as the zero locus of a section of a vector bundle.

 \begin{theorem}[\Cref{thm: isotropic construction}]\label{thm: zero locus intro}
    Let $\overline{r}=(r_1, r_2)$ and $n\geq 0$. There exists a vector bundle $\CV$ on $\CM^{\nc}_{\overline{r}, n} $ together with a section $s$ such that $\CM_{\overline r,n}$ is realised as the zero locus $Z(s)$
    \[
\begin{tikzcd}
& \CV\arrow[d]\\
\CM_{\overline{r}, n}\cong Z(s)\arrow[r, hook, "\iota"] &\CM^{\nc}_{\overline{r}, n}.\arrow[u, bend right, swap, "s"]
\end{tikzcd}
\]
\end{theorem}
The moduli spaces $\CM_{\overline{r}, n}$ and $\CM^{\nc}_{\overline{r}, n} $ are naturally endowed with a torus $\TT$-action (see \Cref{sec: ytorus action}), compatible with the zero locus construction of \Cref{thm: zero locus intro}. Thanks to \Cref{thm: zero locus intro} and by  the work of Behrend-Fantechi \cite{BF_normal_cone}, the gauge origami moduli space $\CM_{\overline{r}, n} $ is endowed with $\TT$-equivariant virtual cycles in homology and $K$-theory
 \begin{align*}
    [\CM_{\overline{r}, n}]^{\vir}&\in A^\TT_{0}\left(\CM_{\overline{r}, n}\right),\\
\oO^{\vir}_{\CM_{\overline{r}, n}}&\in K^\TT_0\left(\CM_{\overline{r}, n}\right).
\end{align*}
Recall  that, if $\overline{r}=(0,r)$, then the moduli space $\CM_{\overline{r}.n}$ is smooth. In this setting, we show in \Cref{prop: caso smooth line ug} that the virtual fundamental class can be explicitly described as
    \begin{align}\label{eqn: intro y chi}
        \oO^{\vir}_{\CM_{\overline{r},n }}=\Lambda_{-t_2}T_{\CM_{\overline{r},n }}^*\in K^{\TT/\TT_{0,2}}_0(\CM_{\overline{r},n } )\left[t_2\right],
    \end{align}
where $T_{\CM_{\overline{r},n }}$ is the tangent bundle of $ \CM_{\overline{r},n }$ and  $\Lambda_{-t_2}(\cdot)$ is the weighted total exterior power. Here,  $t_2$ is the equivariant parameter of a natural (and trivial) $\BC^*$-action\footnote{This $ \BC^*$-action is the restriction of the more general (and non-trivial) torus $\TT$-action on $\CM_{\overline{r},n }$, see \Cref{sec: ytorus action}.} on $\CM_{\overline{r},n }$, see~\Cref{sec: smooth}.

\subsection{Partition function}\label{sec: intro partition function}
The gauge origami moduli space $ \CM_{\overline{r},n }$ is not proper, therefore we cannot directly define invariants via intersection theory. However, it is naturally endowed with the action of torus $\TT$ of rank $2+r_1+r_2$,   whose equivariant parameters we denote by $(t_1, t_2, w_{11}, \dots, w_{2r_{2}})$, see~\Cref{sec: ytorus action}. We define the \emph{gauge origami partition function} as the generating series
\begin{align}\label{eqn: intro inv}
    \CZ_{\overline{r}}(q)=\sum_{n\geq 0} q^n\cdot \chi\left(\mathcal{M}_{\overline{r}, n}, \oO^{\vir}_{\mathcal{M}_{\overline{r}, n}}\right)\in \BQ(t_1, t_2, w)\llbracket q \rrbracket,
\end{align}
where integration is defined via $\TT$-equivariant residues in $K$-theory, see \Cref{sec:invariants}.

In the smooth case $\overline{r}=(0,r)$, by the explicit description of the virtual structure sheaf \eqref{eqn: intro y chi}, the partition function \eqref{eqn: intro inv} reproduces the generating series of the equivariant $\chi_{-t_2}$-genus of $\mathcal{M}_{\overline{r}, n}$, which is readily computed using the localization theorem in $K$-theory.  

Our second main result is a closed formula for the partition function $\CZ_{\overline{r}}(q) $ for all ranks $\overline{r}$. We set $ \CZ^{(1)}(q)=\CZ_{(1,0)}(q)$ and $
    \CZ^{(2)}(q)=\CZ_{(0,1)}(q)$ to be the ''rank 1'' partition functions.
\begin{theorem}[\Cref{thm:factorization}, \Cref{cor: explicit expression inv}]\label{thm: explicit intro}
    Let $\overline{r}=(r_1, r_2)$. We have
    \[
    \CZ_{\overline{r}}(q)=\Exp\left(q\cdot \frac{(1-t_1t_2)(1-t_1^{r_1}t_2^{r_2})}{(1-t_1)(1-t_2)}\right),
    \]
    where $\Exp$ is the plethystic exponential \eqref{eqn: on ple}.
    Moreover, there is a factorisation 
      \begin{align*}
        \CZ_{\overline{r}}(q)=\prod_{\alpha=1}^{r_1}\CZ^{(1)}( q t_1^{r_1-\alpha}t_2^{r_2})\cdot \prod_{\alpha=1}^{r_2}\CZ^{(2)}( q t_2^{r_2-\alpha}).
    \end{align*}
\end{theorem}
The factorisation of the invariants into \emph{rank 1 theories}  has already been observed in a similar context in Donaldson-Thomas theory   \cite{FMR_higher_rank, CKM_crepant, FM_tetra, FT_0, FT_1}, for  motivic invariants \cite{MR_nested_Quot, CR_framed_motivic, CRR_higher_rank, MRhyper, GLMRS_double-nested-1} and  in string theory \cite{NP_colors, dZNPZ_playing_index_M_theory}. See also \cite[Sec.~6.4.7]{KLT_DTPT} where a rank-2-to-rank-1 reduction result is interpreted as a   wall-crossing phenomenon.
\subsubsection{Strategy of the proof}\label{sec: strategy intro}
We briefly explain the main ideas behind the proof of \Cref{thm: explicit intro}. The torus fixed locus  $\CM_{\overline{r}, n}^\TT $ is reduced, zero-dimensional and classified  by tuples $\bn=(\bn_1, \bn_2)$, where each $\bn_i$ is an $r_i$-tuple of non-negative numbers, see \Cref{prop: fixed locus reduced}. Therefore by   the virtual localisation in $K$-theory \cite{FG_riemann_roch, Qu_virtual_pullback},  the gauge origami partition function can be expressed as
\begin{align*}
       \CZ_{\overline{r}}(q)=\sum_{\bn}\mathfrak{e}(-T_{\bn}^{\vir})\cdot q^{\lvert\bn\rvert},
\end{align*}
where $ T_{\bn}^{\vir}$ is the \emph{virtual tangent bundle} at the fixed point corresponding to $\bn$ and $ \lvert\bn\rvert$ denotes its \emph{size}. Each $ \mathfrak{e}(-T_{\bn}^{\vir})$ is a rational function in the $\TT$-equivariant parameters and  should be thought as an \emph{equivariant volume} associated to each $\bn$, which contributes to the partition function $ \CZ_{\overline{r}}(q)$ in a purely combinatorial way, thefore  allowing the partition function to be amenable for computations, as  in the original  framework of the  \emph{vertex formalism}  \cite{MNOP_1}.  The second key point is  of global nature, and asserts that the gauge origami instantons  partition function does not depend on the \emph{framing parameters} $w$, see  \Cref{thm: framinh independence}. This independence is the shadow of a \emph{rigidity principle} and ultimately relies on the properness of the Quot-to-Chow morphism.

Granting this framing independence, we can scale the framing parameters $w$ to infinity, at different \emph{speeds}. A clever choice of such scaling yields a   combinatorial expression of the partition function where the factorisation property appears. Finally, the explicit expression in terms of the plethystic exponential is obtained by combining the factorised form of the partition function with the explicit formula for the \emph{rank 1} case, which is obtained by a much simpler computation performed on the symmetric products of $\BA^1$ combined with a standard combinatorial identity.
\subsubsection{Nekrasov-Okounkov twist}
Nekrasov-Okounkov \cite{NO_membranes_and_sheaves} introduced a \emph{symmetrised} virtual structure sheaf $\widehat{\oO}^{\vir}$, in the context of a conjectural relation between the Donaldson-Thomas theory of  smooth threefolds and the M2-brane index in M-theory. In our case, the \emph{Nekrasov-Okounkov twist} results simply in a twist of the virtual structure sheaf by (the square root of) a torus character
\[
  \widehat{\oO}^{\vir}=\oO^{\vir}\otimes \left(t_1^{r_1}t_2^{r_2}\right)^{-n/2},
\]
see \Cref{prop: twist NO}. Combining the latter identity with \Cref{thm: explicit intro}, we obtain the following closed formula, where we formally set $[x]=x^{1/2}-x^{-1/2}$.
\begin{corollary}[\Cref{cor: NO inv}]
       Let $\overline{r}=(r_1, r_2)$. We have
      \[
  \sum_{n\geq 0} q^n\cdot \chi\left(\mathcal{M}_{\overline{r}, n}, \widehat{\oO}^{\vir}_{\mathcal{M}_{\overline{r}, n}}\right)=\Exp\left(q\cdot  \frac{[t_1t_2][t_1^{r_1} t_2^{r_2}] }{[t_1][t_2]}\right).
    \]
\end{corollary}
\subsubsection{Cohomological limit}
The entire discussion of \Cref{sec: intro partition function} can be formulated in $\TT$-equivariant cohomology, rather than in $\TT$-equivariant $K$-theory, by integrating over the virtual fundamental class $ [\CM_{\overline{r}, n}]^{\vir}$.
\begin{corollary}[\Cref{cor: cohom}]\label{cor: cohom intro}
       Let $\overline{r}=(r_1, r_2)$. We have 
     \[\sum_{n\geq 0}q^n\cdot \int_{[\CM_{\overline{r}, n}]^{\vir}}1=\left(\frac{1}{1-q}\right)^{  \frac{(s_1+s_2)(r_1s_1+r_2s_2)}{s_1s_2}}.\]
\end{corollary}
In the above expression, integration is once again understood $\TT$-equivariantly, where $s_1, s_2$ are the generators of the  $(\BC^*)^2$-equivariant cohomology of the point. While the above result could be proved by adapting the strategy of \Cref{sec: strategy intro} to the setting of equivariant cohomology, it is interesting to remark that the formula in \Cref{cor: cohom intro} can be realised as an appropriate limit of the partition function $ \CZ_{\overline{r}}(q)$. This limiting procedure is a shadow of the \emph{dimensional reduction} phenomenon happening at the level of the partition functions in supersymmetric string theory, see e.g.~the discussion in \cite[Sec.~1.2.2]{FM_tetra}.
\subsubsection{Elliptic genus}
A  common string-theoretic approach in this context is to introduce  \emph{instanton partition functions} from the elliptic genus of the low-energy worldvolume theory on D1-branes for the appropriate brane systems, see e.g.~\cite{BBPT_elliptic_DT, PYZ_tetrahedron}. Mathematically, this can be realised by studying the \emph{virtual elliptic genus} \cite{FG_riemann_roch} of the corresponding moduli space, or their suitable variation, see e.g.~\cite{FMR_higher_rank, FM_tetra,Liu_IMRN_ELL, Ont_critical}. However, as already noticed in \cite{FMR_higher_rank, FM_tetra},  even though the localisation principles allows for exact computation up to any given order,  the elliptic refinement of the partition function in \emph{higher rank} typically depends non-trivially on the framing parameters. This results in the impossibility to replicate the strategy of \ref{sec: strategy intro} to find an exact closed formula for the refined partition function purely in terms of the rank 1 theory.

\subsection{ADHM quiver}
The  Nekrasov's partition function \cite{Nek_instantons} is a fundamental player in the physics of $\CN=2$ four-dimensional gauge theory. Mathematically, it can be realised by performing intersection theory on the smooth moduli space of framed sheaves $\CM^{\fr}_{\overline{r},n}$, where the natural identification comes once more by the natural description of this space as a a Nakajima quiver variety, see \Cref{sec: ADHM}. We show that the gauge origami partition function on broken lines can be obtained from $\CM^{\fr}_{\overline{r},n} $ by integrating the (deformed) Euler class of a tautological bundle $\CF$ and its dual. In string-theoretic terms, we believe this should correspond to the Nekrasov partition function with fundamental and antifundamental matters, with the matter parameters fine-tuned as
\begin{align*}
      \chi\left(\mathcal{M}_{\overline{r}, n}, \oO^{\vir}_{\mathcal{M}_{\overline{r}, n}}\right)&=  \chi\left( \CM^{\fr}_{r,n},  \Lambda_{-t_1}\CF_1^{*}\otimes \Lambda_{-t_2}\CF_2^{*}\otimes  \Lambda_{-t_1t_2}\CF\right),\\
     \int_{[\CM_{\overline{r}, n}]^{\vir}}1&=\int_{\CM^{\fr}_{r,n}}e\left(\CF_1\cdot t_1^{-1} \right)\cdot e\left(\CF_2\cdot t_2^{-1} \right)\cdot e(\CF^*\cdot t_1^{-1}t_2^{-1}),
\end{align*}
see \Cref{cor: ADHM vpull}.
\smallbreak
In a similar fashion, we show that the partition function $\CZ_{\overline{r}}(q)$ coincides with a tautological integral on the Quot scheme $\Quot_{\BA^2}(\oO^{r}, n)$  of zero-dimensional quotients on $\BA^2$
\begin{align*}
      \chi\left(\mathcal{M}_{\overline{r}, n}, \oO^{\vir}_{\mathcal{M}_{\overline{r}, n}}\right)&=  \chi\left( \Quot_{\BA^2}(\oO^{r}, n), \oO^{\vir}_{\Quot_{\BA^2}(\oO^{r}, n)}\otimes \Lambda^\bullet(\CI^{[n]})^{*}\right),\\
     \int_{[\CM_{\overline{r}, n}]^{\vir}}1&=\int_{[\Quot_{\BA^2}(\oO^{r}, n)]^{\vir}}e(\CI^{[n]}),
\end{align*}
see \Cref{cor: v pull formula}, where $\CI^{[n]}$ is a \emph{tautological bundle}, see \Cref{sec: virtu pull back}. In both cases, the proof is obtained by a direct application of the virtual pullback formula \cite{Manolache-virtual-pb, KKP}.
\subsection{Geometry and Physics}
Modern enumerative geometry systematically addressed the mathematical study of 
instanton partition functions  via moduli spaces of sheaves -- in our case, Quot schemes, see e.g.~\cite{Okounk_Lectures_K_theory, Arb_K-theo_surface, FMR_higher_rank,  CKM_K_theoretic, Mon_canonical_vertex, CKM_crepant, Mon_PhD},  providing a geometric interpretation  of  some results initially introduced in topological string theory, see for instance  \cite{IKV_topological_vertex, AK_quiver_matrix_model, BBPT_elliptic_DT, Nek_magnificient_4, NP_colors, ST_crepant} for a non-exhaustive list of references. This work continues the circle of ideas we developed in our previous work \cite{FM_tetra}, where we initiated the program of studying systems of intersecting branes -- as in Nekrasov's seminal work on gauge origami, see also   \cite{PYZ_JHEP, KN_Gauge1, ST_tetra, KN_Gauge2, KN_Gauge3, Kim_Double_quiver} -- via Quot schemes of singular varieties. From this point of view, one should see the present work (resp. \cite{FM_tetra}) as geometrising the instanton partition function from D1-branes probing intersecting D3-branes (resp. D7-branes) in Type IIB. The case of intersecting D5-branes, i.e. Nekrasov's  gauge origami, will appear in a forthcoming\footnote{We thank N. Arbesfeld, M. Kool and W. Lim for sharing an early version of their work.} work of Arbesfeld-Kool-Lim \cite{AKL_origami} which crucially relies on  the novel foundational construction of Oh-Thomas \cite{OT_1}.
\subsubsection{Compact curves}
We conclude by remarking that the enumerative geometry of Quot schemes  on \emph{smooth} projective curves recently manifested a rising interest, thanks to their application to a wide spectrum of problems, ranging from geometric representation theory, quantum $K$-theory and Geometric Langlands, see e.g.~\cite{MO_vafaIntr, OS_Invariants_K-theory, MN_Quot_derived, MN_Quot_Naka, MRhyper, SZ_Quantum1, SZ_Quantum2, Ric_motive_quot_locally_free, Nesterov_Moduli_Quasimap_Higgs}. In view of the derived structure recently constructed on hyperquot schemes in full generality \cite{MPR_derived},  we believe it would be interesting to explore  the geometry and invariants   of a \emph{compact} version of the moduli space $\CM_{\overline{r
},n}$, where we replace  the pair $\CC\subset \BA^2 $ with  $C\subset S$  any boundary divisor with smooth irreducible components inside a smooth projective surface.

\subsection*{Acknowledgments}
We are grateful to Nadir Fasola, Taro Kimura, Martijn Kool, Henry Liu  for useful discussions, and to Aitor Iribar Lopez and Rahul Pandharipande for asking an interesting  question at the Algebraic Geometry and Moduli Seminar at ETH Zürich. We are grateful to the anonymous referees for the useful comments.
 \section{Moduli space}
\subsection{Quot scheme}
 Let $\overline{r}=(r_1, r_2)$ and $n\geq 0$. Recall that we set 
 \[\mathcal{C}=Z(x_1x_2) \subset \BA^2\]
 to be  the union of the two coordinate axes of $\BA^1_i=Z(x_i)\subset \BA^2$. Let $\iota_i: \BA^1_i\hookrightarrow \mathcal{C}$ be the inclusions of the irreducible components, and define the torsion sheaf
 \[\CE_{\overline{r}}=\iota_{1,*}\oO^{{r_1}}_{\BA^1_1}\oplus \iota_{2,*}\oO^{{r_2}}_{\BA^1_2}\]
on $\mathcal{C}$. We defined the \emph{gauge origami the moduli space  on broken lines} as Grothendieck's Quot scheme
\[\CM_{\overline{r}, n}=\Quot_\mathcal{C}(\CE_{\overline{r}},n),\]
which parametrizes flat families of zero-dimensional quotients $[\CE_{\overline{r}}\onto Q]$ on $\mathcal{C}$, where two such quotients are identified whenever their kernels coincide.
 We have a natural closed immersion 
 \begin{align}\label{eqn: embedding Quot in QUot C4}
     \CM_{\overline{r}, n}\hookrightarrow \Quot_{\BA^2}(\oO_{\BA^2}^{r_1+r_2},n),
 \end{align}
 obtained by precomposing the quotients $[\CE_{\overline{r}}\onto Q]$ with \[\oO_{\BA^2}^{r_1+r_2} \onto \CE_{\overline{r}}.\]
 Denote by $ \Sym^{n}\mathcal{C}$ the \emph{$n$-th symmetric power} of $\mathcal{C}$. Attached to $\CM_{\overline{r}, n}$ there is the \emph{Quot-to-Chow morphism} \cite{Rydh1, Fantechi-Ricolfi-structural}
 \begin{align}\label{eqn: QTC}
   \rho: \CM_{\overline{r}, n}\to \Sym^{n}\mathcal{C},   
 \end{align}
 which, at the level of closed points, sends a quotient to its 0-dimensional support, counted with multiplicity.
   \begin{example}\label{ex: smooth case}
    If $\overline{r}=(0,r)$,  the moduli space of gauge origami coincides with the usual Quot scheme on the affine line $\BA^1$
    \begin{align*}
        \CM_{\overline{r}, n}\cong \Quot_{\BA^1}(\oO^{r}, n),
    \end{align*}
    which is a smooth irreducible variety of dimension $nr$, see e.g.~\cite{MR_lissite}. If moreover $r=1$, then 
       \begin{align*}
        \CM_{\overline{r}, n}= \Hilb^n\BA^1\cong  \Sym^n\BA^1\cong \BA^n
    \end{align*}
    is the symmetric product of $n$ points on $\BA^1$.
\end{example}
 \begin{example}\label{example: blow up}
    Let $\overline{r}=(r_1, r_2)$ with $r_1,r_2\geq 1$ and $n= 1$. Recall the embedding \eqref{eqn: embedding Quot in QUot C4} 
 \[
 \CM_{\overline{r},1}\hookrightarrow \Quot_{\BA^2}(\oO_{\BA^2}^{r_1+r_2},1)\cong \BA^2\times \BP^{r_1+r_2-1},
 \]
 where the last isomorphism follows by \cite[Rem.~2.2]{MR_lissite}. Define the closed subvarieties
 \begin{align*}
     Y_i&= \Quot_{\BA^1_i}(\oO^{r_i}_{\BA^1_i},1)\cong \BA^1_i\times \BP^{r_i-1}\hookrightarrow \BA^2\times \BP^{r_1+r_2-1}, \quad i=1,2,\\
     Y_0&= \Quot_{\pt}(\oO_{\pt}^{r_1+r_2}, 1)\cong \{(0,0)\}\times \BP^{r_{1}+r_2-1}\hookrightarrow\BA^2\times \BP^{r_1+r_2-1}.
 \end{align*}
 Then we have that $\CM_{\overline{r},1}$ has three irreducible components\footnote{We remark that  $ \CM_{\overline{r},1}$ seems to share some features of the   moduli space of \emph{expanded degenerations} of length 1 quotients of $\CE_{\overline{r}}$ on $\CC$. In fact, as soon as the support of a quotient $[\oO_{\BA_{i}^1}^{r_i}\onto Q]$ in $Y_i$ moves toward the origin $0\in \mathcal{C}$, a \emph{bubbling} phenomenon occurs, creating the \emph{bubble} $Y_0$, where higher rank quotients occur, yet only supported on $0\in \CC$. We thank the anonymous referee for pointing out that a dimension count suggests
that $ \CM_{\overline{r},1}$ does not coincide with the moduli space parametrizing quotients on degenerations to the normal cone of the origin of $\mathcal{C}$.}, given by $Y_1, Y_2, Y_0$, i.e. 
 \[
 \CM_{\overline{r},1}=Y_1\cup Y_2\cup Y_0.
 \]
 \end{example}
    The situation in Example \ref{example: blow up} readily generalizes to  $n\geq 1$, where a locally closed stratification of $  \CM_{\overline{r}, n}$ is induced by a natural stratification of $\Sym^n\mathcal{C} $, suggesting  an  interpretation of $\CM_{\overline{r}, n}$ as a  degeneration phenomenon.
  
 \subsection{Quiver model}\label{sec: quiver}
Consider  the framed quiver $\LQ$ with one node labelled by $0$ with two loops and two nodes $\infty_1, \infty_2$ with   framing vectors to the node $0$, see  Fig.~\ref{fig:tetrahedron-quiver}.
\begin{figure}[ht]\contourlength{2.5pt}
    \centering\vspace{-3mm}
    \begin{tikzpicture}
    \node(F123) at (0,1){$\infty_1$};
    \node(F234) at (0,-1){$\infty_2$};
    \node(Gk0) at (2.5,0){$0$};
    \draw[->](F123) to[]node{\contour{white}{}} (Gk0);
    \draw[->](F234) to[]node{\contour{white}{}} (Gk0);
    \draw[->](Gk0) to[out=-100,in=-30,looseness=15]node{\contour{white}{}} (Gk0);
        \draw[<-](Gk0) to[out=30,in=100,looseness=15]node{\contour{white}{}} (Gk0);
    \end{tikzpicture}
    \caption{Framed quiver $\LQ$.}
    \label{fig:tetrahedron-quiver}
\end{figure}
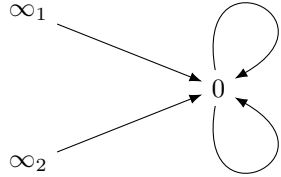

Let $\overline{r}=(r_1, r_2)$ and $n\geq 0$.  A $ \LQ$\emph{-representations} with dimension vector $(\overline{r}, n)$ is the data of:
\begin{itemize}
    \item a tuple of vector spaces $(W_1,W_2,V)$ (one to each node of $\LQ$) such that $\dim V=n$ and $\dim W_i=r_i$ for $i=1,2$,
    \item homomorphisms $B_a\in\End(V)$ and $I_i\in\Hom(W_i,V)$   for $a,i=1, 2$.
\end{itemize}

Let $(W_1,W_2,V)$  be a tuple of  vector spaces with dimension vector $(r_1, r_2, n)$. The moduli space of $\LQ$-representations is the   affine space 
\[
\mathsf R_{\overline r,n}\defeq \End(V)^{\oplus 2}\oplus\bigoplus_{i=1}^2\Hom(W_i,V)
\]
of dimension $(2n^2+(r_1+r_2)n)$. 
An element of $\mathsf R_{\overline r,n}$, together with the choice of $(W_1,W_2,V)$, forms a  $\LQ$-representation.

The group $\GL(V)$ acts naturally on the representation space $\mathsf R_{\overline r,n}$ by simultaneously conjugating the endomorphisms $B_a\in\End(V)$ and scaling the maps $I_i\in \Hom(W_i, V)$. Consider then  the open subscheme 
\[
\mathsf U_{\overline r,n}=\Set{(B_1,B_2,I_1,I_2)\in \mathsf R_{\overline r,n}|\,\mathrm{span}\set{\BC\langle B_1,B_2\rangle I_i(W_i), i=1, 2}\cong V}\subset \mathsf R_{\overline r,n}. 
\]
The $\GL(V)$-action is free on $\mathsf U_{\overline r,n}$, so the quotient 
\[\mathcal M^{\nc}_{\overline r,n}\defeq\mathsf U_{\overline r,n}/\GL(V)\]
exists as a smooth $(n^2+n(r_1+r_2))$-dimensional quasiprojective variety, see~\cite[Lemma~2.1]{Ric_noncomm}. The moduli space $\mathcal M^{\nc}_{\overline r,n}$ is known as the \textit{non-commutative Quot scheme}, and parametrises finite length quotients 
\[[\mathbb C\langle x_1,x_2\rangle^{\oplus (r_1+r_2)}\onto Q],\]
 see~\cite[Thm.~2.5]{BR_higher_rank} and \cite[Prop.~2.3]{Ric_noncomm}.

 \subsection{Zero-locus construction}\label{sec: zero locus}
We prove in this section Theorem \ref{thm: zero locus intro} from the Introduction. Over the affine space $\mathsf R_{\overline r,n}$,  consider the trivial  vector bundle of rank $n^2+(r_1+r_2)n$
\[
\mathcal V\defeq\Lambda^2\BC^2\otimes\End(V)\oplus\bigoplus_{i=1}^2\BC\otimes \Hom(W_i,V).
\]
The vector bundle $\CV$ restricts to the open locus $\mathsf U_{\overline r,n}\subset \mathsf R_{\overline r,n}$ and is naturally $\GL(V)$-equivariant, with respect to the $\GL(V)$-action on $ U_{\overline r,n}$. Therefore, $\CV$ descends to a vector bundle on the quotient $\mathcal M^{\nc}_{\overline r,n}$. With a slight abuse of notation, we keep denoting the induced vector bundle by $\mathcal V$.
 \begin{theorem}\label{thm: isotropic construction}
    Let $\overline{r}=(r_1, r_2)$ and $n\geq 0$. There exists a section $s\in H^0(\CM_{\overline r,n}^{\nc},\CV)$ such that $\CM_{\overline r,n}$ is realised as the zero locus $Z(s)$
    \[
\begin{tikzcd}
& \CV\arrow[d]\\
\CM_{\overline{r}, n}\cong Z(s)\arrow[r, hook, "\iota"] &\CM^{\nc}_{\overline{r}, n}.\arrow[u, bend right, swap, "s"]
\end{tikzcd}
\]
\end{theorem}

\begin{proof}
We start by constructing a section of $\CV$ on the affine space $R_{\overline r,n}$.  Let $(B_1, B_2, I_1, I_2)\in R_{\overline r,n}$ be a quiver representation and  let $(e_1, e_2)$ be a basis of $\BC^2$. Define the section 
\begin{align*}
    s(B_1, B_2, I_1, I_2)=(e_1\wedge e_2)\otimes [B_1, B_2]+\sum_{i=1}^2 B_iI_i,
\end{align*}
where $[B_1, B_2]=B_1B_2-B_2B_1$ denotes the commutator operator.
Since the restriction of the section to the open locus $ \mathsf U_{\overline r,n}$ is $\GL(V)$-equivariant, it descends to a section on the quotient $ \CM^{\nc}_{\overline r,n}$, which  we keep denoting by  $s\in H^0(\CM^{\nc}_{\overline r,n},\CV)$.

There is a natural closed embedding
\begin{align*}
    \iota:\CM_{\overline r,n}\hookrightarrow \CM^{\nc}_{\overline r,n},
\end{align*}
which is obtained by the closed embedding \eqref{eqn: embedding Quot in QUot C4}, followed by the closed embedding of the Quot scheme $\Quot_{\BA^2}(\oO_{\BA^2}^{r_1+r_2},n)$ into its non-commutative counterpart, defined via their moduli functors. We claim that
\[\CM_{\overline r,n}\cong Z(s).\]
Let $[\alpha: \CE_{\overline{r}}\onto Q]\in  \CM_{\overline r,n}$ be a closed point. Recall that the embedding into the non-commutative Quot scheme identifies this quotient with the quiver representation $(B_1, B_2, I_1, I_2)\in  \CM^{\nc}_{\overline r,n}$, where 
\begin{align*}
    &B_i:Q\to Q, \quad i=1,2,\\
    &I_i(a_{ij})=\alpha(\tilde{a}_{ij})\in Q, \quad i=1,2,\quad j=1, \dots, r_i, 
\end{align*}
where $B_i$ is the multiplication by $x_i$, the vectors $(a_{ij})_{j} $ form  a basis of $W_i$ and $\tilde{a}_{ij}$ is the constant 1 in the $j$-th component of $\iota_{i,*}\oO_{\BA_i}^{r_i}\subset \CE_{\overline{r}}$.
Via this correspondence, the matrices $B_i$ are the module multiplication by $x_i$, and the \emph{cyclic vectors} $I_i(a_{ij})$ are the images of the constant 1, each coming from a summand of $\CE_{\overline{r}}$. 

By construction, if $(B_1, B_2, I_1, I_2)\in  \CM^{\nc}_{\overline r,n} $ is a closed point in the image of $\iota$, then it satisfies the relations
\begin{align}\label{eqn: relations}
\begin{cases}
    [B_1, B_2]=0,\\
B_iI_i=0, \quad i=1, 2.
\end{cases}
\end{align} 

On the other hand, consider a quotient  $[\mathbb C\langle x_1,x_2\rangle^{\oplus (r_1+r_2)}\onto Q] \in \CM^{\nc}_{\overline r,n}$ satisfying the relations \eqref{eqn: relations}. Then the quotient factors through
\[
\begin{tikzcd}
    \mathbb C\langle x_1,x_2\rangle^{\oplus (r_1+r_2)} \arrow[r,two heads]\arrow[d,two heads]& Q\\
  \CE_{\overline{r}} \arrow[ur, two heads]&   
\end{tikzcd}
\]
thus defining a closed point in $\CM_{\overline r,n}$, as desired.
We conclude by noticing that  the same argument holds over flat families, which yields the desired isomorphism.
\end{proof}
\begin{remark}
   We chose to include the (trivial) factors $\Lambda^2 \BC^2\cong \BC$   and $\BC$ in the definition of $\CV$ in order to keep track of the correct  $\TT$-equivariance  of $\CV$ which we will need in \Cref{sec: vertex}, see e.g.~\eqref{eqn: correct lambda 2}.
\end{remark}
\subsection{Torus representations and their weights}
Let $\TT = (\BC^*)^g$ be an algebraic torus, with character lattice $\widehat{\TT}\cong \BZ^g$. Let $K^0_{\TT}(\pt)$ be the Grothendieck group of the category of $\TT$-representations. Any finite-dimensional $\TT$-representation $V$ splits as a sum of $1$-dimensional representations called the \emph{weights} of $V$. Each weight corresponds to a character $\mu \in \widehat{\TT}$, and in turn each character corresponds to a monomial $\tf^\mu = \tf_1^{\mu_1}\cdots \tf_g^{\mu_g}$ in the coordinates of $\TT$. In other words, there is an isomorphism 
\begin{equation}\label{eqn: equiv}
K^0_{\TT}(\pt) \cong \BZ \left[\tf^\mu \mid \mu \in \widehat{\TT}\right],
\end{equation}
identifying the class of a $\TT$-representation with  its decomposition into weight spaces\footnote{The notation $\BZ\left[\tf^\mu \mid \mu \in \widehat{\TT}\right]$ indicates the span over $\BZ$ of the elements $t^\mu$, not the algebra generated by $t^\mu$ as symbols.}. We will therefore sometimes identify a (virtual) $\TT$-representation with its character. 
\subsection{Torus action}\label{sec: ytorus action}
Let $\overline{r}=(r_1, r_2)$  and define the algebraic tori
\[
\TT=\TT_0\times \TT_1, \quad \TT_0=(\BC^*)^2, \quad \TT_1=(\BC^*)^{r_1+r_2}.
\]
The torus $\TT_0$ acts on $\BA^2$ by
\[(t_1, t_2)\cdot (x_1, x_2)=(t_1x_1, t_2 x_2).\]
 In particular, it restricts to the union of the coordinate axes $\mathcal{C}\subset \BA^2$ and naturally lifts to the gauge origami moduli space  $\CM_{\overline{r}, n}$, by moving the support of the quotients. The torus $\TT_1$  acts on $\CM_{\overline{r}, n}$ by scaling the factors of $\CE_{\overline{r}}$. Therefore, we have an induced $\TT$-action on $\CM_{\overline{r}, n}$. Equivalently, the  $\TT$-action   on $ \CM_{\overline{r}, n}$ is the restriction of the  $\TT$-action on the moduli space of non-commutative Quot scheme $ \CM^{\mathrm{nc}}_{\overline{r}, n}$ given on points by
 \begin{align}\label{eqn: torus action on smooth amb}
   (t, w)\cdot (B_1, B_2,  I_1, I_2)=(t_1^{-1} B_1, t_2^{-1} B_2, \overrightarrow{w}_{1}^{-1}I_1, \overrightarrow{w}_{2}^{-1}I_{2}),  
 \end{align}
 where $w_{ij}$ are the coordinates of $\TT_1$, for $i=1, 2$ and $j=1\dots, r_i$ and we set
 \[
\overrightarrow{w}_{i}=\begin{pmatrix}
w_{i,1} & 0   & \cdots & 0   & 0 \\
0   & \ddots & \ddots & \vdots & 0 \\
\vdots & \ddots & \ddots & \ddots & \vdots \\
0   & \cdots & \ddots & \ddots & 0 \\
0   & 0   & \cdots & 0   & w_{i,r_i}
\end{pmatrix},
 \]
 for $i=1,2$. We remark that the vector spaces $W_i$ from \Cref{sec: quiver} decompose equivariantly as $W_i=\bigoplus_{\alpha=1}^{r_i}w_{i\alpha}$.
\begin{remark}
    If $\overline{r}=(0, r)$, recall from \Cref{ex: smooth case} that $        \CM_{\overline{r}, n}\cong \Quot_{\BA^1_2}(\oO^{r}, n)$. In this case, the subtorus 
    \begin{align*}
        \TT_{0,2}=\langle t_2 \rangle\subset \TT_0
    \end{align*}
    acts trivially on the gauge origami moduli space   $\CM_{\overline{r}, n}$. However, its perfect obstruction theory is not $\TT_{0,2}$-trivial, as we will explicitly see in \Cref{prop: caso smooth line ug}. Therefore we have
    \begin{align*}
        K_\TT^0(\CM_{\overline{r},n } )\cong   K_{\TT/\TT_{0,2}}^0(\CM_{\overline{r},n } )\left[t_2^{\pm 1}\right].
    \end{align*}
\end{remark}
The torus fixed locus of the gauge origami moduli space   is classified by tuples of non-negative integers.
    \begin{prop}\label{prop: fixed locus reduced}
    Let $\overline{r}=(r_1, r_2)$ and $n\geq 0$. The $\TT$-fixed locus $\mathcal{M}_{\overline{r}, n}^\TT$ is reduced, zero-dimensional and in bijection with
    \[\set{\bn=(\bn_1, \bn_2)| \bn_i=(n_{i1},\dots, n_{ir_i}),  \mbox{ such that } n=\sum_{i=1}^2\sum_{j=1}^{r_i}n_{ij} \mbox{ and } n_{ij}\geq 0 \mbox{ for } i=1,2,\,\, j=1, \dots, r_i}.\]
    \end{prop}
    \begin{proof}
        By  the same arguments as in \cite{Bifet} its $\TT_1$-fixed locus is scheme-theoretically isomorphic to 
        \begin{align}\label{eqn: fixed blaa}
            \mathcal{M}_{\overline{r}, n}^{\TT_1}\cong \coprod_{n=\sum_{i=1}^2\sum_{j=1}^{r_i}n_{ij}}\prod_{i=1}^2\prod_{j=1}^{r_i}\Quot_{\mathcal{C}}(\iota_{i, *}\oO_{\BA^1_i},n_{ij}).
        \end{align}
More explicitly, a $\TT_1$-fixed quotient in $\mathcal{M}_{\overline{r}, n}^{\TT_1}$ splits as
    \[[\CE_{\overline{r}}\onto Q]\cong \left[\bigoplus_{i=1}^2\bigoplus_{j=1}^{r_i} \iota_{i,*}\oO_{\BA^1_i}\onto\bigoplus_{i=1}^2\bigoplus_{j=1}^{r_i}Q_{ij}\right],\]
       with the quotient map in the right-hand-side being diagonal. The identification \eqref{eqn: fixed blaa}  associates  to a $\TT_1$-fixed quotient $[\CE_{\overline{r}}\onto Q]\in \mathcal{M}_{\overline{r}, n}^{\TT_1} $ the tuple $ ([\iota_{i,*}\oO_{\BA^1_i}\onto Q_{ij}])_{i,j}\in \prod_{i=1}^2\prod_{j=1}^{r_i}\Quot_{\mathcal{C}}(\iota_{i, *}\oO_{\BA^1_i},n_{ij})  $ for some suitable non-negative integers $n_{ij}$.

    Moreover, we have an identification of each factor with a Hilbert scheme of points \[\Quot_{\mathcal{C}}(\iota_{i, *}\oO_{\BA^1_i},n_{ij})\cong \Hilb^{n_{ij}}(\BA_i^1), \] such that  the natural $\BC^*$-action on $\BA^1_i$ lifts on $\Hilb^{n_{ij}}(\BA_i^1)$ and coincides with the induced $\TT_0$-action on the left-hand-side.
The $ \BC^*$-fixed locus $  \Hilb^{n_{ij}}(\BA_i^1)^{\BC^*}$ consists of the single isolated reduced point corresponding to the monomial ideal
\[
I_{ij}=(x_{\hat{i}}^{n_{ij}})\BC[x_{\hat{i}}]\subset \BC[x_{\hat{i}}],
\]
where, to ease the notation, we adopt the convention for the indices $\set{i, \hat{i}}=\set{1,2}$.
Therefore  
 $\mathcal{M}_{\overline{r}, n}^\TT$ is reduced and 0-dimensional, and  every $\TT$-fixed point in $ \mathcal{M}_{\overline{r}, n}^\TT$ corresponds to a quotient
        \[\bigoplus_{i=1}^2\bigoplus_{j=1}^{r_i}I_{ij}\hookrightarrow \bigoplus_{i=1}^2 \iota_{i,*}\oO_{\BA^1_i}^{r_i}\onto\bigoplus_{i=1}^2\bigoplus_{j=1}^{r_i}\oO_{Z_{ij}},\]
        where each $Z_{ij}\subset \BC^1_{i}$ is the unique 0-dimensional subscheme supported at the origin of length $n_{ij}$.
         Reversing the correspondence concludes the argument.
    \end{proof}
    Given a pair $\bn=(\bn_1, \bn_2)$ as in \Cref{prop: fixed locus reduced}, we denote the \emph{size} of $\bn$ by $\lvert\bn\rvert=\sum_{i=1}^2\sum_{j=1}^{r_i}n_{ij}$.
    \begin{corollary}
    Let $\overline{r}=(r_1, r_2)$.   The generating series of topological Euler characteristics is
        \[
        \sum_{n\geq 0}e(\mathcal{M}_{\overline{r}, n})     \cdot q^n =  \frac{1}{(1-q)^{r_1+r_2}}.
        \]
    \end{corollary}
    \begin{proof}
        The topological Euler characteristic of a $\BC$-scheme  coincides with the one of its $\TT$-fixed locus, therefore the result follows by Proposition \ref{prop: fixed locus reduced}.
    \end{proof}
\subsection{Virtual fundamental classes}\label{sec: virtual classes}
We construct in this section of a virtual fundamental class on  $\CM_{\overline{r}, n}$ exploiting the quiver moduli description of \Cref{thm: isotropic construction}. 

We briefly recall the classical construction of Behrend-Fantechi \cite{BF_normal_cone} in a simpler global toy model. Let a scheme $Z$
\[
\begin{tikzcd}
& \CE\arrow[d]\\
Z:=Z(s)\arrow[r, hook, "\iota"] &\CA\arrow[u, bend right, swap, "s"]
\end{tikzcd}
\]
be the zero locus of an  section $s\in \Gamma(\CA,\CE)$, where $\CE$ is a rank $r$ vector bundle over a smooth quasi-projective variety $\CA$. 
Then there exists an induced obstruction theory on $Z$
\begin{equation}\label{eqn: obs th}
 \begin{tikzcd}
 \big[  \CE^*|_Z\arrow[d, "s"]\arrow[r]&\Omega_{\CA|Z}\big]\arrow[d, phantom, "\parallel"]\arrow[r, phantom, "\cong"]&\BE\arrow[d]\\
    \big[\CI/\CI^2\arrow[r]&\Omega_{\CA|Z}\big]\arrow[r, phantom, "\cong"]&\BL_Z,
\end{tikzcd}
\end{equation}
in $\derived^{[-1,0]}(Z)$, where $\CI \subset \oO_{A}$ is the ideal sheaf of the inclusion $Z \into \CA$ and $\BL_Z$ denotes the truncated cotangent complex. By the work of \cite{BF_normal_cone}, there exist a \emph{virtual fundamental class} and a \emph{virtual structure sheaf} on $Z$ satisfying
\begin{align*}
    \iota_*[Z]^{\vir}&=e(\CE)\cap [\CA]\in A_{\dim \CA- r}\left(\CA\right),\\
\iota_*\oO_Z^{\vir}&=\mathfrak{e}(\CE)\otimes [\oO_{\CA}]\in K_0\left(\CA\right),
\end{align*}
where we set
\begin{align}\label{eqn: K th EUler}
   \mathfrak{e}(V):=\sum_{i\geq 0}(-1)^i\Lambda^i V^* \in K^0(\CA)
\end{align}
to be the total exterior power of the dual of a vector bundle $V$ on $\CA$, as a class in $K$-theory. We call the integer $\dim \CA-r$ the \emph{virtual dimension} of $Z$. Moreover, the Behrend-Fantechi construction naturally extends to the torus equivariant setting.

Since all ingredients  in Theorem \ref{thm: isotropic construction} are naturally $\TT$-equivariant, we  immediately obtain the following corollary.
\begin{corollary}\label{cor: virtual classes}
   The gauge origami moduli space  $\CM_{\overline{r}, n}$ is endowed with a $\TT$-equivariant virtual fundamental class and a $\TT$-equivariant virtual structure sheaf of virtual dimension zero
   \begin{align*}
    [\CM_{\overline{r}, n}]^{\vir}&\in A_{0}^\TT\left(\CM_{\overline{r}, n}\right),\\
\oO^{\vir}_{\CM_{\overline{r}, n}}&\in K_0^\TT\left(\CM_{\overline{r}, n}\right).
\end{align*}
\end{corollary}
Recall from \Cref{ex: smooth case} that, if $\overline{r}=(0,r)$, then the moduli space $\CM_{\overline{r},n}$ is smooth. In this setting, we show in \Cref{prop: caso smooth line ug} that the ($\BC^*$-equivariant) virtual fundamental class can be explicitly described as
    \begin{align*}
        \oO^{vir}_{\CM_{\overline{r},n }}=\Lambda_{-t_2}T_{\CM_{\overline{r},n }}^*\in K_0(\CM_{\overline{r},n } )\left[t_2\right],
    \end{align*}
  where $T_{\CM_{\overline{r},n }}$ is the tangent bundle of $ \CM_{\overline{r},n }$ and $t_2$ is the equivariant parameter of the trivial $\BC^*$-action on $\CM_{\overline{r},n }$, see~\Cref{sec: smooth}.

 \section{The partition function}

\subsection{Virtual invariants}\label{sec:invariants}
For any  $\overline{r}=(r_1, r_2)$ and $n\in \BZ_{\geq 0}$, denote by $t_1, t_2$ the irreducible characters of the $\TT_0$-action on $\mathcal{M}_{\overline{r}, n}$ and by $w_{ij}$ the irreducible characters of the $\TT_1$-action, where $i=1, 2$ and $j=1,\dots r_i$.  We define the $\TT$-equivariant $K$-theoretic invariants
\begin{align}\label{eqn: invariants}
\CZ_{\overline{r},n}=\chi\left(\mathcal{M}_{\overline{r}, n}, \oO^{\vir}_{\mathcal{M}_{\overline{r}, n}}\right)\in \BQ(t_1, t_2, w).
\end{align}
Since $\CM_{\overline{r}, n}$ is not proper, one cannot directly define $K$-theoretic invariants via proper push-forward in $K$-theory. Instead,  since the fixed locus $\CM_{\overline{r}, n}^\TT $ is proper, we define\footnote{alternatively, one can define the invariants directly by virtual localisation, analogously to \Cref{sec: cohom}.} invariants as the composition
\begin{align*}
  \chi\left(\mathcal{M}_{\overline{r}, n}, \cdot\right):  K^{\TT}_0\left(\CM_{\overline{r}, n}\right)\to K^{\TT}_0(\CM_{\overline{r}, n})_{\mathrm{loc}}\xrightarrow{\sim} K^{\TT}_0(\CM^{\TT}_{\overline{r}, n})_{\mathrm{loc}}\to K_0^\TT(\pt)_{\mathrm{loc}}.
\end{align*}
Here, the first map is a suitable localisation of $ K^{\TT}_0\left(\CM_{\overline{r}, n}\right)$, the second map is  Thomason's abstract localisation \cite{Tho:formule_Lefschetz} and the third map is the proper pushforward on  $\CM_{\overline{r}, n}^\TT $ (extended to the localisation).

\smallbreak

We define the \emph{gauge origami partition function} as the generating series
\begin{align*}
    \CZ_{\overline{r}}(q)=\sum_{n\geq 0} \CZ_{\overline{r},n}\cdot q^n \in \BQ(t_1, t_2, w)[\![q]\!].
\end{align*}
Let $T_{\CM_{\overline{r}, n}}^{\vir}$ be the \emph{virtual tangent bundle} of $\mathcal{M}_{\overline{r}, n}$, that is  the dual of the obstruction theory $\BE$ in \eqref{eqn: obs th}
\begin{align*}
    T_{\CM_{\overline{r}, n}}^{\vir}=[  {T_{\CM^{\mathrm{nc}}_{\overline{r}, n} }}|_{\mathcal{M}_{\overline{r}, n}}\to \CV|_{\mathcal{M}_{\overline{r}, n}}].
\end{align*}
We will denote by $ T_{\CM_{\overline{r}, n}}^{\vir}$ also its class in $K$-theory, whenever it is  clear from the context.

By the virtual localisation theorem in K-theory  \cite{FG_riemann_roch, Qu_virtual_pullback} we can compute the invariants via $\TT$-equivariant residues on the fixed locus
\begin{align}\label{eqn: loc K}
    \chi\left(\mathcal{M}_{\overline{r}, n}, \oO^{\vir}_{\mathcal{M}_{\overline{r}, n}}\right)=\sum_{\lvert\bn\rvert=n}\frac{1}{\mathfrak{e}(T_{\bn}^{\vir})},
\end{align}
where the sum runs over all pairs $\bn=(\bn_1, \bn_2)$ of size $n$ as in \Cref{prop: fixed locus reduced}, $ T_{\bn}^{\vir}$ is the virtual tangent bundle at the fixed point corresponding to $\bn$ via the correspondence of \Cref{prop: fixed locus reduced} and $\mathfrak{e}$ is the equivariant analogue of \eqref{eqn: K th EUler}. More precisely, the operator $\mathfrak{e}$ acts on a virtual $\TT$-representation as 
\[
\mathfrak{e}\left(\sum_{\mu}t^{\mu}-\sum_{\nu}t^{\nu}\right)=\frac{\prod_\mu(1-t^{-\mu})}{\prod_\nu(1-t^{-\nu})},
\]
where $t^{\mu}, t^{\nu}$ are irreducible $\TT$-representations such that no $t^{\nu}$ are  the trivial representation.

In particular,  we implicitly   used in \eqref{eqn: loc K} the fact that the virtual tangent bundle is $\TT$-movable\footnote{Given a virtual $\TT$-representation $V\in K^0_{\TT}(\pt)$, we denote by $V^{\mathrm{fix}}$ the constant term of $V$ under the identification \eqref{eqn: equiv}, or, equivalently, the virtual  $\TT$-subrepresentation corresponding to the trivial weight.  The \emph{$\TT$-movable part} of $ V$ is  $V^{\mathrm{mov}}=V-V^{\mathrm{fix}}$ and we say that $V$ is \emph{ $\TT$-movable} if $V^{\mathrm{fix}}=0. $} by Proposition \ref{prop: t movable}. Summing over all $n$, we express the gauge origami partition function as
\begin{align*}
       \CZ_{\overline{r}}(q)=\sum_{\bn}\mathfrak{e}(-T_{\bn}^{\vir})\cdot q^{\lvert\bn\rvert}.
\end{align*}
\subsection{Vertex formalism}\label{sec: vertex}
Let $r=(r_1, r_2)$ and $n\geq 0$,  and let  $\bn$ correspond to a $\TT$-fixed quotient   $[\CE_{\overline{r}}\onto Q_{\bn}]\in \CM_{\overline{r}, n}^\TT$, with notation as in \Cref{prop: fixed locus reduced}. We denote by $T^{\vir}_{\bn}$ the virtual tangent space at the fixed point $\bn$. Recall that $Q_{\bn}$ decomposes as 
\[Q_{\bn}=\bigoplus_{i=1}^2\bigoplus_{j=1}^{r_i}\oO_{Z_{{n_{ij}}}}, \]
where each $Z_{n_{ij}}\subset\BA^1_i$ is the unique 0-dimensional closed subscheme of length $n_{ij}$ supported at $0$.  We keep denoting their  global sections by $Z_{n_{ij}}$ (resp. $Q_{\bn}$)  which  are given, as a $\TT_0$ (resp. $\TT$)-representation, by
\begin{align}\label{eqn: decomposition of Q}
\begin{split}
    Z_{n_{ij}}&=\sum_{a=0}^{n_{ij}-1}t_{\hat{i}}^{a},  \quad i=1, 2, \quad j=1, \dots, r_i, \\
Q_{\bn_i}&=\sum_{j=1}^{r_i}w_{ij}Z_{n_{ij}}, \quad i=1, 2,\\
Q_{\bn}&=\sum_{i=1}^2Q_{\bn_i}.
\end{split}
\end{align}
 We denote by $\overline{(\cdot)}$ the involution on $K^0_\TT(\pt)$ which acts as $\overline{t^\mu}=t^{-\mu}$ on irreducible representations.
Finally, set the $\TT_1$-representations
\begin{align*}
    K_i&=\sum_{j=1}^{r_i}w_{ij}, \quad i=1,2,\\
    K&=\sum_{i=1}^2 K_i.
\end{align*}
\begin{prop}\label{prop: t vir}
    We have an identification of virtual  $\TT$-representations
    \begin{align*}        T^{\vir}_{\bn}= \sum_{i=1}^2\overline{K_i}(1-t^{-1}_i) Q_{\bn}-(1-t_1^{-1})(1-t_2^{-1})Q_{\bn}\overline{Q_{\bn}}.
    \end{align*}
\end{prop}
\begin{proof}
    The virtual tangent space at a $\TT$-fixed point corresponding to a tuple $\bn$ is computed as a class in $K$-theory by
    \begin{align}\label{eqn: virtual complex}
T^{\vir}_{\bn}=T_{\CM^{\mathrm{nc}}_{\overline{r}, n} }|_{\bn}-\CV|_{\bn}\in K_\TT^0(\pt).
    \end{align}
    The tangent space $T_{\CM^{\mathrm{nc}}_{\overline{r}, n} }|_{\bn} $ at the fixed quotient corresponding to  $\bn\in \CM^{\mathrm{nc}}_{\overline{r}, n} $ sits in an exact sequence of $\TT$-representations
    \begin{align*}
        0\to \Hom(Q_{\bn}, Q_{\bn})\to \BC^2\otimes \Hom(Q_{\bn},Q_{\bn})\oplus \bigoplus_{i=1}^2\bigoplus_{j=1}^{r_i}\Hom(\BC \cdot w_{ij}, Q_{\bn})\to T_{\CM^{\mathrm{nc}}_{\overline{r}, n} }|_{\bn}\to 0,
    \end{align*}
    where the middle representation  is the tangent space of the affine space $\mathsf R_{\overline r,n} $ and the first representation records the free action of $\GL(\BC^n)$. The weight decomposition of $\BC^2$ is
    \[\BC^2=t_1^{-1}+t_2^{-1},\]
   by which one readily computes in $K$-theory
    \[
    T_{\CM^{\mathrm{nc}}_{\overline{r}, n} }|_{\bn}=(t_1^{-1}+t_2^{-1}-1)Q_{\bn}\overline{Q_{\bn}}+ \overline{K}Q_{\bn}.
    \]
   The vector bundle $\CV$ at the point $\bn$ satisfies
   \[
\CV|_{\bn}=\Lambda^2\BC^2\otimes \Hom(Q_{\bn}, Q_{\bn})\oplus\bigoplus_{i=1}^2\BC\otimes \Hom(K_i, Q_{\bn}),
   \]
   where the weight decomposition of $ \Lambda^2\BC^2$ is 
   \begin{align}\label{eqn: correct lambda 2}
 \Lambda^2\BC^2=t_1^{-1}t_2^{-1},
   \end{align}
   and the one of each appropriate copy of $\BC$ is
   \[
   \BC=t_i^{-1},
   \]
   which implies that 
   \[
\CV|_{\bn}=t_1^{-1}t_2^{-1}Q_{\bn}\overline{Q_{\bn}}+\sum_{i=1}^2\overline{K_i}t^{-1}_i Q_{\bn}.
   \]
   Plugging the computations in \eqref{eqn: virtual complex} concludes the proof.
\end{proof}
From now on, we will refer to $T^{\vir}_{\bn}$ as the \emph{vertex term} associated to $\bn$. 
\begin{remark}
    As remarked in the proof of \Cref{prop: caso smooth line ug}, the same identity of \Cref{prop: t vir} holds \emph{globally},  by replacing $Q_{\bn}$ with the universal sheaf $\CQ$ on $\CM_{\overline{r}, n}\times \BA^2$.
\end{remark}
\subsubsection{Factorisation}\label{sec: fact subsec} The vertex term $T^{\vir}_{\bn}$ can be further expressed as a sum
\begin{align*}
T^{\vir}_{\bn}&=\sum_{i,j=1}^2\mathsf{v}^{(ij)}_{\bn},\\
\mathsf{v}^{(ij)}_{\bn}&=
\overline{K_i}(1-t_i^{-1}) Q_{\bn_j}-(1-t_1^{-1})(1-t_2^{-1}) \overline{Q_{\bn_i}} Q_{\bn_j}, \quad i=1,2.
\end{align*}
By \eqref{eqn: decomposition of Q} we can further identify each terms $ \mathsf{v}^{(ij)}_{\bn}$ as sums
\begin{align*}
\mathsf{v}^{(ij)}_{\bn}&=\sum_{\substack{1\leq \alpha\leq r_i\\ 1\leq \beta\leq r_j}}\mathsf{v}^{(ij, \alpha\beta)}_{\bn},\quad i,j=1,2,\\
    \mathsf{v}^{(ij, \alpha\beta)}_{\bn}&=
     w_{i\alpha}^{-1}w_{j\beta}\left((1-t_i^{-1})Z_{n_{j\beta}}-(1-t_1^{-1})(1-t_2^{-1})\overline{Z_{n_{i\alpha}}}Z_{n_{j\beta}}\right), \quad \alpha=1, \dots, r_i, \quad \beta=1, \dots, r_j.
\end{align*}
As already anticipated, the vertex term is $\TT$-movable, which in particular implies that the rational function $\mathfrak{e}(-T^{\vir}_{\bn})$ is well-defined and non-zero.
\begin{prop}\label{prop: t movable}
    The vertex term $T^{\vir}_{\bn}$ is $\TT$-movable.
\end{prop}
\begin{proof}
By the decomposition of the vertex term as
\begin{align}\label{eqn: decomposition vertex}
T^{\vir}_{\bn}=\sum_{i,j=1}^2\sum_{\substack{1\leq \alpha\leq r_i\\ 1\leq \beta\leq r_j}}\mathsf{v}^{(ij, \alpha\beta)}_{\bn},
\end{align}
it is clear that the possible contributions to the $\TT$-fixed part of $T^{\vir}_{\bn}$ have to come from the    \emph{diagonal} terms $\mathsf{v}^{(ii, \alpha\alpha)}_{\bn}$.  A simple algebraic manipulation yields the identity
\begin{align*}
    \mathsf{v}^{(ii, \alpha\alpha)}_{\bn}=(1-t_i^{-1})\sum_{a=1}^{n_{i\alpha}}t_{\hat{i}}^{-a},
\end{align*}
where we set $\set{i, \hat{i}}=\set{1,2}$, 
which confirms that $  \mathsf{v}^{(ii, \alpha\alpha)}_{\bn}$ is $\TT$-movable.
\end{proof}
\subsection{Quot scheme of a smooth line}\label{sec: smooth}  Recall from \Cref{ex: smooth case} that in the case $\overline{r}=(0, r)$ the gauge origami moduli space  $\CM_{\overline{r}, n}$ ($\TT$-equivariantly) coincides with the usual smooth Quot scheme $\Quot_{\BA_2^1}(\oO^{r}, n)$. In this case, we can express the induced $\TT$-equivariant virtual cycles on $\CM_{\overline{r}, n}$ more explicitly.

Let $y$ be a variable and $V$ a vector bundle on a scheme $X$. We set
\[
\Lambda_y V=\sum_{i\geq 0} y^i \Lambda^i V\in K^0(X)[y].
\]
\begin{prop}\label{prop: caso smooth line ug}
    Let $\overline{r}=(0,r)$ and $n\geq 0$. Then we have an identity
    \begin{align*}
        T^{\vir}_{\CM_{\overline{r}, n}}=T_{\CM_{\overline{r},n }}-t_2^{-1}\cdot T_{\CM_{\overline{r},n }}\in K_{\TT/\TT_{0,2}}^0(\CM_{\overline{r},n } )\left[t_2^{\pm 1}\right], 
    \end{align*}
    where $T_{\CM_{\overline{r},n }}$ is the tangent bundle of $ \CM_{\overline{r},n }$. In particular, we have that
    \begin{align*}
        \oO^{vir}_{\CM_{\overline{r},n }}=\Lambda_{-t_2}T_{\CM_{\overline{r},n }}^*\in K^{\TT/\TT_{0,2}}_0(\CM_{\overline{r},n } )\left[t_2\right].
    \end{align*}
\end{prop}
\begin{proof}
Denote by $\CQ$ the universal quotient sheaf on $ \CM_{\overline{r},n }\times \BA^2$. By the very same proof of \Cref{prop: t vir}, we have an identity in $\TT$-equivariant $K$-theory
\begin{align*}        T^{\vir}= \sum_{i=1}^2\overline{K_i}(1-t^{-1}_i) \CQ-(1-t_1^{-1})(1-t_2^{-1})\CQ\overline{\CQ}.
    \end{align*}

    Exploiting the standard quiver description of  $\Quot_{\BA_2^1}(\oO^{r}, n)$,  the tangent space of  $ \CM_{\overline{r},n }$  can be equivariantly expressed in $\TT$-equivariant $K$-theory as 
    \begin{align*}
T_{\CM_{\overline{r}, n}}=  \overline{K_2} \CQ-(1-t_1^{-1})\CQ\overline{\CQ},     
    \end{align*}
    from which the first claim follows. For the second claim, since $\CM_{\overline{r},n }$ is smooth, we can identify it as the zero locus of the zero section of the $\TT$-equivariant vector bundle $ t_2^{-1}\cdot T_{\CM_{\overline{r},n }}$. The dual of the induced perfect obstruction theory on $\CM_{\overline{r},n }$  has class  $T_{\CM_{\overline{r},n }}-t_2^{-1}\cdot T_{\CM_{\overline{r},n }}$ in $K$-theory.  By \cite{Tho_K-theo_Fulton}, the induced virtual structure sheaf only depends on the class in $K$-theory of the perfect obstruction theory, which implies
    \begin{equation}\label{eqn: y deformation}
    \begin{split}
  \oO^{vir}_{\CM_{\overline{r},n }}&=\mathfrak{e}\left(t_2^{-1}\cdot T_{\CM_{\overline{r},n }}\right)\\
 &= \Lambda_{-t_2}T_{\CM_{\overline{r},n }}^*.    
    \end{split}
    \end{equation}
\end{proof}
As a corollary of \Cref{prop: caso smooth line ug}, we obtain that for $\overline{r}=(0, r)$ the $K$-theoretic invariant $\CZ_{\overline{r},n} $ coincides with the $\TT$-equivariant \emph{$\chi_{-y}$}-genus of $\CM_{\overline{r},n }$, where the role of $y$ is played by the equivariant parameter $t_2$. 
\begin{corollary}\label{cor: chi y}
   Let $\overline{r}=(0, r)$ and $n\geq 0$. Then there is an identity 
\begin{align*}
\CZ_{\overline{r},n}=\chi(\CM_{\overline{r},n },  \Lambda_{-t_2}T_{\CM_{\overline{r},n }}^*).   
\end{align*}
\end{corollary}
\subsection{Framing limits}\label{sec:framing}
We prove in this section Theorem \ref{thm: explicit intro} from the Introduction.  We first show that the partition function $\CZ_{\overline{r}}(q)$ does not depend on the framing parameters $w_{ij}$,  exploiting a suitable rigidity argument. Granting this independence, we  specialise the framing parameters $w_{ij}$ to arbitrary values and send them to infinity, which recovers the formula in Theorem \ref{thm: explicit intro}.

\begin{theorem}\label{thm: framinh independence}
    The gauge origami partition function $\CZ_{\overline{r}}(q)$ does not depend on the weights $w_{i\alpha}$, for $i=1, 2$ and $\alpha=1, \dots, r_i$.
\end{theorem}
\begin{proof}
    We prove this  result  by exploiting similar arguments to \cite[Thm. 6.5]{FMR_higher_rank}.    The $n$-th coefficient of $\CZ_{\overline{r}}(q)$ is a sum of contributions
\[
\mathfrak{e}(-T^{\vir}_{\bn}),\qquad \lvert \bn\rvert = n.
\]
By the definition of the operator $\mathfrak{e}(\cdot)$ we have that
\begin{equation*}\label{eqn:poles_showing_up}
\mathfrak{e}(-T^{\vir}_{\bn}) = A(t)\prod_{(i,\alpha)\neq (j,\beta)} \frac{\prod_{\mu_{ij,\alpha\beta}}(1-w_{i\alpha}^{-1}w_{j\beta} t^{\mu_{ij,\alpha\beta}})}{\prod_{\nu_{ij,\alpha\beta}}(1-w_{i\alpha}^{-1}w_{j\beta} t^{\nu_{ij,\alpha\beta}})},
\end{equation*}
where the product is over all $(i,\alpha)\neq (j,\beta) $ such  that $i,j=1, 2$ and $\alpha=1, \dots, r_i$ and $\beta=1, \dots, r_j$ while  $A(t)\in \BQ(\!(t_1, t_2)\!) $ and the number of weights $\mu_{ij,\alpha\beta}$ and $\nu_{ij,\alpha\beta}$ is the same. Thus, $\CZ_{\overline{r}}(q)$ is a homogeneous rational expression of total degree 0 with respect to the variables $w_{i\alpha}$. We aim to show that none of the possible poles of the form $1-w_{i\alpha}^{-1}w_{j\beta} t^{\nu}$ appear in  $\CZ_{\overline{r}}(q)$,  implying that the latter  is a degree $0$ polynomial in the $w_{i\alpha}$, hence constant in the $w_{i\alpha}$.

Recall that an irreducible $\TT$-representation $\mathsf w$ is called \emph{compact} if the fixed locus of the torus  $\TT_{\mathsf w} = \ker (\mathsf w)$ is proper and \emph{non-compact} otherwise, see~\cite[Def.~3.1]{Arb_K-theo_surface}.

Set $\mathsf w =  w_{i\alpha}^{-1}w_{j\beta}t^\nu$ for fixed $(i,\alpha)\neq (j,\beta)$ and $t^\nu$ an irreducible $\TT_0$-representation. By  \cite[Prop.~3.10]{FM_tetra} (see also \cite[Prop.~3.2]{Arb_K-theo_surface}), if $ \mathsf w$ is compact weight, then  $1-\mathsf w$ cannot be a pole of  $\CZ_{\overline{r}}(q)$. We show now that $ \mathsf w$ is indeed a compact weight.

Consider the automorphism $\tau_\nu\colon \TT \simto \TT$ defined by
\[
(t_1, t_2,w_{11}, \dots, w_{2r_2}) \mapsto (t_1,t_2,w_{11},\ldots,w_{i\alpha}t^{-\nu},\ldots,w_{j\beta},\ldots,w_{2r_2}).
\]
It maps $\TT_{\mathsf w}\subset \TT$ isomorphically onto the subtorus $\TT_0 \times \set{w_{i\alpha}=w_{j\beta}} \subset \TT$. This yields a closed inclusion of tori
\begin{equation}\label{inclusions_tori}
\TT_0 \simto \TT_0 \times \Set{(1,\ldots,1)} \into \tau(\TT_{\mathsf w}).
\end{equation}
We consider the action $\sigma_\nu \colon \TT \times \CM_{\overline{r}, n} \to \CM_{\overline{r}, n} $ where $\TT_0$ translates the support of the quotient sheaf in the usual way, the $\alpha$-th summand of $\iota_{i,*}\OO^{\oplus r_i}$ gets scaled by $w_{i\alpha}t^\nu$ and all other summands by $w_{j\beta}$ for $(j,\beta)\neq (i,\alpha)$. 
 We have a commutative diagram
\[
\begin{tikzcd}[row sep=large]
\TT_{\mathsf w} \times \CM_{\overline{r}, n} \arrow{r}{\sigma}\arrow[swap]{d}{\tau_\nu\times\id} & \CM_{\overline{r}, n} \arrow[equal]{d} \\
\tau_\nu(\TT_{\mathsf w}) \times \CM_{\overline{r}, n} \arrow{r}{\sigma_\nu} & \CM_{\overline{r}, n} 
\end{tikzcd}
\]
where $\sigma$ is the restriction of the usual $\TT$-action on $\CM_{\overline{r}, n}$. This diagram induces a natural isomorphism $\CM_{\overline{r}, n} ^{\TT_{\mathsf w}}\simto\CM_{\overline{r}, n} ^{\tau_\nu(\TT_{\mathsf w})}$, which combined with \eqref{inclusions_tori}
yields an inclusion
\[
\CM_{\overline{r}, n} ^{\TT_{\mathsf w}}\simto\CM_{\overline{r}, n} ^{\tau_\nu(\TT_{\mathsf w})} \into \CM_{\overline{r}, n} ^{\TT_0},
\]
where $\CM_{\overline{r}, n} ^{\TT_0}$ is the fixed locus with respect to the action $\sigma_\nu$. By the properness of the Quot-to-Chow morphism \eqref{eqn: QTC}, the fixed locus  $\CM_{\overline{r}, n} ^{\TT_0}$is proper, since a $\TT_0$-fixed surjection $\CE_{\overline{r}, n}\onto Q$ necessarily has the quotient $Q$ entirely supported at the origin $0 \in \mathcal{C}$, see also \cite[Rem.~3.1]{FM_tetra}. Thus $\mathsf w$ is a compact irreducible $\TT$-representation, and the result follows.
\end{proof}
By Theorem \ref{thm: framinh independence}  the gauge origami partition function s does not depend on the framing parameters $w_{i\alpha}$, therefore we can compute it by specialising them to arbitrary values, similarly to \cite[Sec.~3.5.2]{FM_tetra}. We totally order  the indices $(i,\alpha)$ lexicographically, setting $(j,\beta)>(i,\alpha)$ if $j>i$ or $i=j$ and $\beta>\alpha$. Then, pick $ N_{i\alpha}\gg 0$ such that $ N_{j\beta}\gg N_{i\alpha}$ for all $(j,\beta)>(i, \alpha)$. Finally, we set $w_{i\alpha}=L^{N_{i\alpha}}$, and consider the limits for  $L\to \infty$.

\begin{prop}\label{prop: limits}
    Let $(j,\beta)>(i,\alpha)$. We have
    \begin{align*}
  & \lim_{L\to \infty}\mathfrak{e}(-\mathsf{v}^{(ij, \alpha\beta)}_{\bn})|_{w_{i\alpha}=L^{N_{i\alpha}}}    =1,\\
     &    \lim_{L\to \infty}\mathfrak{e}(-\mathsf{v}^{(ji, \beta\alpha)}_{\bn})|_{w_{i\alpha}=L^{N_{i\alpha}}}=t_j^{n_{i\alpha}}.
    \end{align*}
\end{prop}
\begin{proof}
      For all monomials $t^{\mu}$, we have the limit
      \begin{align*}
          \lim_{L\to \infty}\mathfrak{e}(w_{i\alpha}^{-1}w_{j\beta}t^{\mu})|_{w_{i\alpha}=L^{N_{i\alpha}}}&= \lim_{L\to \infty}(1-L^{N_{i\alpha}-N_{j\beta}}t^{-\mu} ) \\
          &=1.
      \end{align*}
 Recall the definition of $ \mathsf{v}^{(ij, \alpha\beta)}_{\bn}$ from \Cref{sec: fact subsec}  and set 
 \begin{align*}
   Z_{n_{i\alpha}}&=\sum_{\nu}t^{\nu},\\
   Z_{n_{j\beta}}&=\sum_{\mu}t^{\mu}.
 \end{align*}
We have
 \begin{multline*}
     \lim_{L\to \infty}\mathfrak{e}(-\mathsf{v}^{(ij, \alpha\beta)}_{\bn})|_{w_{i\alpha}=L^{N_{i\alpha}}}  =\\
     \lim_{L\to \infty}\frac{\prod_\mu(1-L^{N_{i\alpha}-N_{j\beta}}t_it^{-\mu})}{\prod_\mu(1-L^{N_{i\alpha}-N_{j\beta}}t^{-\mu})}\cdot \frac{\prod_{\nu,\mu}(1-L^{N_{i\alpha}-N_{j\beta}}t^{\nu}t^{-\mu})\cdot \prod_{\nu,\mu}(1-L^{N_{i\alpha}-N_{j\beta}}t_1t_2t^{\nu}t^{-\mu})    }{\prod_{\nu,\mu}(1-L^{N_{i\alpha}-N_{j\beta}}t_1t^{\nu}t^{-\mu})\cdot \prod_{\nu,\mu}(1-L^{N_{i\alpha}-N_{j\beta}}t_2t^{\nu}t^{-\mu})}=1.
 \end{multline*}
Similarly, we have
 \begin{multline*}
     \lim_{L\to \infty}\mathfrak{e}(-\mathsf{v}^{(ji, \beta\alpha)}_{\bn})|_{w_{i\alpha}=L^{N_{i\alpha}}}  =\\
     \lim_{L\to \infty}\frac{\prod_\nu(1-L^{N_{j\beta}-N_{i\alpha}}t_jt^{-\nu})}{\prod_\nu(1-L^{N_{j\beta}-N_{i\alpha}}t^{-\nu})}\cdot \frac{\prod_{\nu,\mu}(1-L^{N_{j\beta}-N_{i\alpha}}t^{\mu}t^{-\nu})\cdot \prod_{\nu,\mu}(1-L^{N_{j\beta}-N_{i\alpha}}t_1t_2t^{\mu}t^{-\nu})    }{\prod_{\nu,\mu}(1-L^{N_{j\beta}-N_{i\alpha}}t_1t^{\mu}t^{-\nu})\cdot \prod_{\nu,\mu}(1-L^{N_{j\beta}-N_{i\alpha}}t_2t^{\mu}t^{-\nu})}\\
     =\frac{\prod_\nu t_jt^{-\nu}\cdot \prod_{\nu,\mu}t^{\mu}t^{-\nu}\cdot \prod_{\nu,\mu}t_1t_2t^{\mu}t^{-\nu} }{\prod_\nu t^{-\nu}\cdot \prod_{\nu,\mu}t_1t^{\mu}t^{-\nu}\cdot \prod_{\nu,\mu}t_2t^{\mu}t^{-\nu} }=t_j^{n_{i\alpha}}.
     \end{multline*} 
\end{proof}
 We define the partition functions
\begin{align*}
    &\CZ^{(1)}(q)=\CZ_{(1,0)}(q),\\
    &\CZ^{(2)}(q)=\CZ_{(0,1)}(q),
\end{align*}
which  by \Cref{cor: chi y} and \Cref{ex: smooth case} are the generating series of equivariant $\chi_{-y}$-genera of the symmetric products $\Sym^{\bullet}\BA^1_i$ for $i=1,2$. Notice that, by symmetry, we have
\[
\CZ^{(2)}(q)=\CZ^{(1)}(q)|_{t_1=t_2, t_2=t_1}.
\]
 For any formal power series 
 \[f(p_1, \ldots, p_r; q_1, \ldots, q_s)\in   \BQ(p_1, \ldots, p_r)[\![q_1, \ldots, q_s]\!],\]
 such that $f(p_1, \ldots, p_r;0,\ldots,0)=0$, its \emph{plethystic exponential} is defined as 
\begin{align}\label{eqn: on ple} 
\Exp(f(p_1, \ldots, p_r;q_1, \ldots, q_s)) &:= \exp\Big( \sum_{n=1}^{\infty} \frac{1}{n} f(p_1^n, \ldots, p_r^n;q_1^n, \ldots, q_s^n) \Big),
\end{align}
viewed as an element of $\BQ(p_1, \ldots, p_r)[\![q_1, \ldots, q_s]\!]$. Recall that the plethystic exponential enjoys the useful identities
\begin{align*}
 \Exp(q+p)&=\Exp(q)\cdot \Exp(p),\\
    \Exp(q)&=\frac{1}{1-q}.
\end{align*}
\begin{prop}\label{prop: rank 1 Exp}
We have
\[
\CZ^{(1)}(q)=\Exp\left(q\cdot \frac{1-t_1t_2}{1-t_2}\right).
\]
\end{prop}
\begin{proof}
By \Cref{ex: smooth case} and \Cref{prop: caso smooth line ug}, we have that $\CM_{(1,0),n}\cong \BA^n$, and its only fixed point $p_n$ has tangent space
\[
T_{p_n}=\sum_{a=1}^n t_2^{-a}.
\]
Therefore by localisation we have
\begin{align*}
    \CZ^{(1)}(q)&=\sum_{n\geq 0} q^n\cdot \mathfrak{e}(t_1^{-1}\cdot T_{p_n}-T_{p_n})\\
    &=\sum_{n\geq 0} q^n\cdot \prod_{a=1}^n\frac{1-t_1t_2^a}{1-t_2^a}\\
    &=\Exp\left(q\cdot \frac{1-t_1t_2}{1-t_2}\right),
\end{align*}
where in the last line we used a standard combinatorial identity, see e.g.~\cite[Ex.~5.1.22]{Okounk_Lectures_K_theory}.
\end{proof}
We show now that the partition function $\CZ_{\overline{r}}(q)$ is governed by its \emph{rank 1} case.
\begin{theorem}\label{thm:factorization}
    Let $\overline{r}=(r_1, r_2)$. We have the factorisation 
    \begin{align*}
        \CZ_{\overline{r}}(q)=\prod_{\alpha=1}^{r_1}\CZ^{(1)}( q t_1^{r_1-\alpha}t_2^{r_2})\cdot \prod_{\alpha=1}^{r_2}\CZ^{(2)}( q t_2^{r_2-\alpha}).
    \end{align*}
\end{theorem}
\begin{proof}
   Set $w_{i\alpha}=L^{N_{i\alpha}}$. By Theorem \ref{thm: framinh independence} the partition function $\CZ_{\overline{r}}(q)$ can be computed in the limit $L\to \infty$. By  Proposition \ref{prop: limits} and \eqref{eqn: decomposition vertex} we have
\begin{align*}
 \CZ_{\overline{r}}(q) &= \lim_{L\to \infty}\sum_{\bn} q^{\lvert \bn \rvert}\prod_{i,j=1}^2 \prod_{\substack{1\leq \alpha\leq r_i\\ 1\leq \beta\leq r_j}} \mathfrak{e}(-\mathsf{v}^{(ij, \alpha\beta)}_{\bn})|_{w_{i\alpha}=L^{N_{i\alpha}}}   \\
     &= \lim_{L\to \infty}\sum_{\bn}\left(\prod_{i=1}^2 \prod_{\alpha=1}^{r_i} q^{n_{i\alpha}}\mathfrak{e}(-\mathsf{v}^{(ii, \alpha\alpha)}_{\bn})\right)\cdot\prod_{(i,\alpha)<(j,\beta)}  \mathfrak{e}(-\mathsf{v}^{(ij, \alpha\beta)}_{\bn})\mathfrak{e}(-\mathsf{v}^{(ji, \beta\alpha)}_{\bn})|_{w_{i\alpha}=L^{N_{i\alpha}}}\\
     &= \sum_{\bn}\left(\prod_{i=1}^2 \prod_{\alpha=1}^{r_i} q^{n_{i\alpha}}\mathfrak{e}(-\mathsf{v}^{(ii, \alpha\alpha)}_{\bn})\right)\cdot\prod_{(i,\alpha)<(j,\beta)}t_j^{n_{i\alpha}}.
    \end{align*}
Rearranging the family of indices obeying $(i, \alpha)<(j, \beta)$ we obtain 
\begin{align*}
\prod_{(i,\alpha)<(j,\beta)}t_j^{n_{i\alpha}}&=\prod_{i=1}^2\prod_{\alpha=1}^{r_i}\left(\prod_{(i, \alpha)<(j, \beta)}t_j^{n_{i\alpha}} \right)\\
&=\prod_{\alpha=1}^{r_1}\left(t_1^{r_1-\alpha}t_2^{r_2}\right)^{n_{1\alpha}} \prod_{\alpha=1}^{r_2}t_2^{(r_2-\alpha)n_{2\alpha}}.
\end{align*}
Therefore we have
\begin{align*}
     \CZ_{\overline{r}}(q) &= \sum_{\bn}\prod_{\alpha=1}^{r_1}\mathfrak{e}(-\mathsf{v}^{(11, \alpha\alpha)}_{\bn})\left(q t_1^{r_1-\alpha}t_2^{r_2}\right)^{n_{1\alpha}}\prod_{\alpha=1}^{r_2}\mathfrak{e}(-\mathsf{v}^{(22, \alpha\alpha)}_{\bn})\left(q t_2^{r_2-\alpha}\right)^{n_{2\alpha}}\\
  &=\prod_{\alpha=1}^{r_1}\sum_{n_{1\alpha}\geq 0} \mathfrak{e}(-\mathsf{v}^{(11, \alpha\alpha)}_{\bn})\left(q t_1^{r_1-\alpha}t_2^{r_2}\right)^{n_{1\alpha}}\cdot \prod_{\alpha=1}^{r_2}\sum_{n_{2\alpha}\geq 0} \mathfrak{e}(-\mathsf{v}^{(22, \alpha\alpha)}_{\bn})\left(q t_2^{r_2-\alpha}\right)^{n_{2\alpha}} \\
 &=\prod_{\alpha=1}^{r_1}\CZ^{(1)}( q t_1^{r_1-\alpha}t_2^{r_2})\cdot \prod_{\alpha=1}^{r_2}\CZ^{(2)}( q t_2^{r_2-\alpha}).
\end{align*}
\end{proof}
Combining Theorem \ref{thm:factorization} with \Cref{prop: rank 1 Exp}, we prove a closed formula for the partition function of gauge origami. 
\begin{corollary}\label{cor: explicit expression inv}
    Let $\overline{r}=(r_1, r_2)$. We have
    \[
    \CZ_{\overline{r}}(q)=\Exp\left(q\cdot \frac{(1-t_1t_2)(1-t_1^{r_1}t_2^{r_2})}{(1-t_1)(1-t_2)}\right).
    \]
\end{corollary}
\begin{proof}
We have the identity
\begin{align*}
    \sum_{\alpha=1}^{r_1}qt_1^{r_1-\alpha}t_2^{r_2}\frac{1-t_1t_2}{1-t_2}+\sum_{\alpha=1}^{r_2}qt_2^{r_2-\alpha}\frac{1-t_1t_2}{1-t_1}&= 
    q(1-t_1t_2)\left( t_2^{r_2}\frac{1-t_1^{r_1}}{(1-t_2)(1-t_1)}+\frac{1-t_2^{r_2}}{(1-t_1)(1-t_2)}\right)\\
    &= q(1-t_1t_2)\cdot \frac{1-t_1^{r_1}t_2^{r_2}}{(1-t_1)(1-t_2)}.
\end{align*}
The claim follows by  combining the above identity with \Cref{thm:factorization} and the multiplicativity of $\Exp$.
\end{proof}
As a second corollary, we obtain a vanishing result for the gauge origami partition function in the Calabi-Yau limit $t_1t_2=1$.
\begin{corollary}
    Let $\overline{r}=(r_1,r_2)$. Then for $n>0$ there is a vanishing
    \[
    \CZ_{\overline{r},n}|_{t_1=t_2^{-1}}=0.
    \]
\end{corollary}
\subsection{Nekrasov-Okounkov twist}
We compute in this section  a slight variation of the invariants defined in \eqref{eqn: invariants}. Let $\BE$ be the complex of the obstruction theory on $\CM_{\overline{r}, n}$ and define $K_{\vir}=\det\BE $ to be the \emph{virtual canonical bundle} of $\CM_{\overline{r}, n}$. We set\footnote{To be precise, for the square root $ K_{\vir}^{1/2}$ to exists $\TT$-equivariant we need to replace $\TT$ with a suitable double cover.}
\begin{align*}
    \widehat{\oO}^{\vir}=\oO^{\vir}\otimes K_{\vir}^{1/2}\in K^{\TT}_0\left(\CM_{\overline{r}, n}, \BZ\left[\tfrac{1}{2}\right] \right)
\end{align*}
to be the \emph{twisted virtual structure sheaf} after Nekrasov-Okounkov \cite{NO_membranes_and_sheaves}. Notice that, a priori, the square root $  K_{\vir}^{1/2}$ may not exists as a genuine line bundle, but as a class in $K$-theory after inverting 2, see \cite[Sec.~5.1]{OT_1}.
\begin{prop}\label{prop: twist NO}
    Let $\overline{r}=(r_1, r_2)$ and $n\geq 0$. We have an identity of $\TT$-equivariant line bundles 
    \[
    K_{\vir}^{1/2}\cong \oO\otimes \left(t_1^{r_1}t_2^{r_2}\right)^{-n/2}.
    \]
   In particular, we have
    \[
       \widehat{\oO}^{\vir}=\oO^{\vir}\otimes \left(t_1^{r_1}t_2^{r_2}\right)^{-n/2}.
    \]
\end{prop}
\begin{proof}
The   claimed isomorphism of line bundle follows by checking it at the level of torus fixed points,  so let $\bn$ correspond to a fixed point in $\CM_{\overline{r}, n}^\TT$. A simple computation yields 
    \begin{align*}
      \det  T^{\vir}_{\bn}=t_1^{nr_1}t_2^{nr_2},
    \end{align*}
    by which we conclude the proof.
\end{proof}
We define the partition function 
\[
\widehat{\CZ}_{\overline{r}}(q)=\sum_{n\geq 0}q^n\cdot \chi\left(\mathcal{M}_{\overline{r}, n}, \widehat{\oO}^{\vir}_{\mathcal{M}_{\overline{r}, n}}\right).
\]
Set the operator 
\[[x]=x^{1/2}-x^{-1/2}.\]
  By \Cref{cor: explicit expression inv} we immediately obtain a closed formula for $\widehat{\CZ}_{\overline{r}}(q)$.
\begin{corollary}\label{cor: NO inv}
       Let $\overline{r}=(r_1, r_2)$. We have
      \[
    \widehat{\CZ}_{\overline{r}}(q)=\Exp\left(q\cdot  \frac{[t_1t_2][t_1^{r_1} t_2^{r_2}] }{[t_1][t_2]}\right).
    \]
\end{corollary}
\subsection{Cohomological limit}\label{sec: cohom}
Let $\overline{r}=(r_1, r_2)$ and $n\geq 0$. By Corollary \ref{cor: virtual classes}, the gauge origami moduli space  is endowed with a $\TT$-equivariant  \emph{virtual fundamental class} $[\CM_{\overline{r}, n}]^{\vir}\in A^\TT_{*}\left(\CM_{\overline{r}, n}\right)$. We define the $\TT$-equivariant invariants  
\begin{align*}\label{eqn: cohom invariants}
\CZ^{\coh}_{\overline{r},n}=\int_{[\CM_{\overline{r}, n}]^{\vir}}1 \in \BQ(s_1, s_2, v),
\end{align*}
where $s_1, s_2, v_{11}, \dots, v_{2r_2}$ are the generators of $\TT$-equivariant Chow cohomology $A^\TT_*(\pt)$. In other words, they can be realized as $s_i=c_1(t_i), v_{i\alpha}=c_1(w_{i\alpha})$, i.e.  the first Chern classes of the irreducible representations generating $K^0_\TT(\pt) $.

Since $\CM_{\overline{r}, n}$ is not proper, one cannot directly define  invariants via proper push-forward in equivariant cohomology. Instead,  since the fixed locus $\CM_{\overline{r}, n}^\TT $ is proper, we define invariants via $\TT$-equivariant residues on the fixed locus
\begin{align*}
   \int_{[\CM_{\overline{r}, n}]^{\vir}}1=\sum_{\lvert\bn\rvert=n}\frac{1}{e(T_{\bn}^{\vir})},
\end{align*}
by Graber-Pandharipande virtual localization formula \cite{GP_virtual_localization}, 
where $e(\cdot)$ is the ($\TT$-equivariant) Euler class.

We define the \emph{cohomological gauge origami partition function} as the generating series
\begin{align*}
    \CZ^{\coh}_{\overline{r}}(q)=\sum_{n\geq 0}  q^n \cdot   \int_{[\CM_{\overline{r}, n}]^{\vir}}1\in\BQ(s_1, s_2, v)[\![q]\!].
\end{align*}
Therefore, the gauge origami partition function is given by
\begin{align*}
     \CZ^{\coh}_{\overline{r}}(q)=\sum_{\bn}e(-T^{\vir}_{\bn})\cdot q^{\lvert \bn\rvert},
\end{align*}
where the sum runs over tuples $\bn$ as in \Cref{prop: fixed locus reduced}.

We recall  how the Euler class $e(\cdot)$ acts on a virtual $\TT$-representation.  For an irreducible representation $t_1^{\mu_1} t_2^{\mu_2}w_{11}^{\mu_{11}}\cdots w_{2r_2}^{\mu_{2r_2}}$, we have
\[e(t_1^{\mu_1} t_2^{\mu_2}w_{11}^{\mu_{11}}\cdots w_{2r_2}^{\mu_{2r_2}})=\mu_1s_1 +\mu_2s_2+\mu_{11}v_{11}+\dots + \mu_{2r_2}v_{2r_2}, \]
and satisfies $e(t^{\mu}\pm t^{\nu})=e(t^{\mu})e(t^{\nu})^{\pm 1}$, as long as $\nu$ is not the trivial weight. To ease the notation we set $\mu=(\mu_1, \mu_2, \mu_{11}, \dots, \mu_{2r_2})$ and $s=(s_1, s_2, v_{11}, \dots, v_{2r_2})$, and write $e(t^{\mu})=\mu\cdot s$ where the product is the usual scalar product.

As explained
in \cite[Sec.~7.1]{FMR_higher_rank}, one should think of $e(\cdot)$ as the \emph{linearization} of $\mathfrak{e}$, since
\begin{align*}
    \mathfrak{e}(t^{\mu})|_{t_i=e^{bs_i}, w_{i\alpha}=e^{bv_{i\alpha}}}&=
  (1-e^{-b\mu\cdot s})\\
  &=be(t^\mu)+O(b^2).
\end{align*}
In particular, if $V$ is a virtual $\TT$-representation of rank 0, the limit $b\to 0$ is well-defined and yields
\begin{align}\label{eqn: lim K to cohom}
    \lim_{b\to 0}\mathfrak{e}(V)|_{t_i=e^{bs_i}, w_{i\alpha}=e^{bv_{i\alpha}}}=e(V).
\end{align}
Since the vertex term  $\mathsf{v}_{\overline{\pi}} $ has rank 0, this identity tells us that the cohomological gauge origami partition function $ \CZ^{\coh}_{\overline{r}}(q) $ is a limit\footnote{Alternatively, we could show that the cohomological invariants are the limit of the $K$-theoretic invariants using virtual  Riemann-Roch \cite{FG_riemann_roch}, see e.g.~\cite[Thm.~6.4]{CKM_crepant}.  } of the gauge origami instanton partition function $ \CZ_{\overline{r}}(q)$.
\begin{corollary}\label{cor: cohom}
     Let $\overline{r}=(r_1, r_2)$. We have 
     \[\CZ^{\coh}_{\overline{r}}(q)=\left(\frac{1}{1-q}\right)^{  \frac{(s_1+s_2)(r_1s_1+r_2s_2)}{s_1s_2}}.\]
\end{corollary}
\begin{proof}
    By the limit \eqref{eqn: lim K to cohom}, we can compute the cohomological instanton partition function as
    \[
    \CZ^{\coh}_{\overline{r}}(q)=\lim_{b\to 0}\CZ_{\overline{r}}(q)|_{t_i=e^{bs_i}}.
    \]
 Therefore we have
    \begin{align*}
        \lim_{b\to 0}\CZ_{\overline{r}}(q)|_{t_i=e^{bs_i}}&= \lim_{b\to 0}\exp\left(\sum_{n\geq 1}\frac{q^n}{n}\frac{(1-e^{-bns_1}e^{-bns_2})(1-e^{-bnr_1s_1}e^{-bnr_2s_2})}{(1-e^{-bns_1})(1-e^{-bns_2})}\right)\\
        &= \exp \left(\sum_{n\geq 1} \frac{q^n}{n} \frac{(s_1+s_2)(r_1s_1+r_2s_2)}{s_1s_2}\right)\\
        &=\left(\frac{1}{1-q}\right)^{  \frac{(s_1+s_2)(r_1s_1+r_2s_2)}{s_1s_2}}.
    \end{align*}
\end{proof}
\section{ADHM quiver}
We explain in this section how the gauge origami moduli space  on broken lines -- and its partition function -- relates to the original ADHM quiver construction and to the Quot scheme $\Quot_{\BA^2}(\oO^r, n)$.
\subsection{Quot scheme of $\BA^2$}\label{sec: cut out in QUot}
Let $r\geq 1$ and $n\geq 0$. The Quot scheme $ \Quot_{\BA^2}(\oO^r, n)$ is an irreducible quasi-projective scheme of dimension $(r+1)n$  and is  singular for $ n\geq 2$, see \cite{EL_irreducibility,MR_lissite}.
Nevertheless, it is endowed by a perfect obstruction theory of virtual dimension $rn$, whose construction goes along the same lines of \Cref{sec: zero locus} and  we now explain.

\subsubsection{Quiver model}
Consider the quiver $ \tilde{\LQ}$ in \Cref{fig:Quot-quiver}.
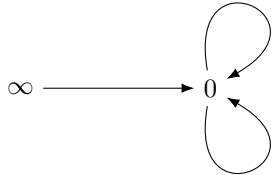
\begin{figure}[ht]\contourlength{2.5pt}
    \centering\vspace{-3mm}
    \begin{tikzpicture}
    \node(F123) at (0,0){$\infty$};
    \node(Gk0) at (2.5,0){$0$};
    \draw[->](F123) to[]node{\contour{white}{}} (Gk0);
    \draw[->](Gk0) to[out=-100,in=-30,looseness=15]node{\contour{white}{}} (Gk0);
        \draw[<-](Gk0) to[out=30,in=100,looseness=15]node{\contour{white}{}} (Gk0);
    \end{tikzpicture}
    \caption{Framed quiver ${\tilde{\LQ}}$.}
    \label{fig:Quot-quiver}
\end{figure}

Fix a dimension vector $(r,n)$ and two vector spaces $(W,V)$ of dimensions $\dim W=r$ and $\dim V=n$.  Similarly to \Cref{sec: quiver},  we define the space
\[
\mathsf U_{r,n}=\Set{(B_1,B_2,I)|\,\BC\langle B_1,B_2\rangle I(W)\cong V}\subset   \End(V)^{\oplus 2}\oplus\Hom(W,V),
\]
which admits a free $\GL(V)$-action, so that the quotient 
\[\mathcal M^{\nc}_{r,n}\defeq\mathsf U_{\overline r,n}/\GL(V)\]
exists\footnote{To be precise, there is an evident isomorphism $\mathcal M^{\nc}_{r,n}\cong \mathcal M^{\nc}_{\overline{r},n} $ for all tuples $(r_1, r_2)$ with $r_1+r_2=r$.} as a smooth $(n^2+nr)$-dimensional quasiprojective variety, see~\cite[Lemma~2.1]{Ric_noncomm}. Define the  trivial  vector bundle of rank $n^2$
\[
\widetilde{\mathcal{V}}\defeq\Lambda^2\BC^2\otimes\End(V)
\]
on $\mathsf U_{r,n} $, 
which by $\GL(V)$-equivariance descends to a vector bundle on $\mathcal M^{\nc}_{r,n}$.
 \begin{prop}\label{prop:zero sec Quot}
    Let $r\geq 1$ and $n\geq 0$. There exists a section $s\in H^0(\CM_{ r,n}^{\nc},\widetilde{\CV})$ such that $\Quot_{\BA^2}(\oO^r, n)$ is realised as the zero locus $Z(s)$
    \[
\begin{tikzcd}
& \widetilde{\CV}\arrow[d]\\
\Quot_{\BA^2}(\oO^r, n) \cong  Z(s)\arrow[r, hook, "\Tilde{\iota}"] &\CM^{\nc}_{r, n}.\arrow[u, bend right, swap, "s"]
\end{tikzcd}
\]
\end{prop}
\begin{proof}
The proof is completely analogous to the one of \Cref{thm: isotropic construction}, with the section $s$ being in this case
\begin{align*}
    s(B_1, B_2, I)=(e_1\wedge e_2)\otimes [B_1, B_2],
\end{align*}
for $(B_1, B_2, I)\in \CM_{\overline r,n}$.
\end{proof}
Notice that $\Quot_{\BA^2}(\oO^r, n) $ is naturally endowed with a $\TT$-action, descending from the $\TT$-action \eqref{eqn: torus action on smooth amb}. By the methods of \Cref{sec: virtual classes} and \Cref{prop:zero sec Quot}, $\Quot_{\BA^2}(\oO^r, n)$ is therefore naturally endowed with a virtual fundamental class and a virtual structure sheaf of virtual dimension $nr$
  \begin{align*}
    [\Quot_{\BA^2}(\oO^r, n)]^{\vir}&\in A^\TT_{nr}\left(\Quot_{\BA^2}(\oO^r, n)\right),\\
\oO^{\vir}_{\Quot_{\BA^2}(\oO^r, n)}&\in K^\TT_0\left(\Quot_{\BA^2}(\oO^r, n)\right).
\end{align*}
\subsubsection{Virtual pullback}\label{sec: virtu pull back}
Let $\overline{r}=(r_1,r_2)$ and set $r=r_1+r_2$. We compare now the virtual structures on $\CM_{\overline{r},n}$ and $\Quot_{\BA^2}(\oO^{r}, n) $, via the closed embedding
\[
  j:  \CM_{\overline{r}, n}\hookrightarrow \Quot_{\BA^2}(\oO_{\BA^2}^{r},n).
\]
\smallbreak
Consider the universal quotient $[\oO^{r}\onto \CQ] $ on $ \Quot_{\BA^2}(\oO_{\BA^2}^{r},n)\times \BA^2$ and denote by $\pi, p$ the natural projections respectively to the first and second factor. For $i=1,2$  consider the short exact sequences
\begin{align}\label{eqn: inclusion of the ideal}
0 \to \CI_i \to \oO_{\BA^2}\to \oO_{\BA^1_i}\to 0,    
\end{align}
and define the vector bundle $\CI=\left(\CI_1^{r_1}\oplus \CI_2^{r_2}\right)^*$ of rank $r$ on $\BA^2$. The vector bundle $\CI$ decomposes $\TT$-equivariantly as
\begin{align*}
    \CI=\bigoplus_{i=1}^2\bigoplus_{\alpha=1}^{r_i}\oO_{\BA^2}\cdot   w_{i,\alpha}^{-1}t_i^{-1}.
\end{align*}
We define the \emph{tautological vector bundle}\footnote{The sheaf $\CI^{[n]}$ is locally free by a standard application of cohomology and base change, exploiting  the fact that $\CQ$ has relative dimension $0$ over $\Quot_{\BA^2}(\oO_{\BA^2}^{r},n)$.} of rank $rn$ on $ \Quot_{\BA^2}(\oO_{\BA^2}^{r},n)$ by
\begin{align*}
    \CI^{[n]}=\pi_*(p^*\CI\otimes \CQ).
\end{align*}

 \begin{prop}\label{prop:zero sec delta in Quot}
    Let $r=(r_1, r_2) $ and $n\geq 0$. There exists a section $s\in H^0(\Quot_{\BA^2}(\oO_{\BA^2}^{r},n), \CI^{[n]})$ such that $\CM_{\overline r,n}$ is realised as the zero locus $Z(s)$
    \[
\begin{tikzcd}
& \CI^{[n]}\arrow[d]\\
\CM_{\overline{r},n} \cong  Z(s)\arrow[r, hook, "j"] &\Quot_{\BA^2}(\oO^r, n).\arrow[u, bend right, swap, "s"]
\end{tikzcd}
\]
\end{prop}
\begin{proof}
The closed embedding $j$ is naturally realised by identifying both $ \CM_{\overline{r},n} $ and $\Quot_{\BA^2}(\oO^r, n) $ inside the smooth ambient space $ \CM^{\nc}_{r, n}$ by \Cref{thm: isotropic construction} and \Cref{prop:zero sec Quot}.

Let $p^*\CI^*\into \oO^{r}$ be the natural inclusion induced by \eqref{eqn: inclusion of the ideal} and define the section $\sigma$ as the composition 
 \[
 \sigma: p^*\CI^*\into \oO^{r}\onto \CQ,
 \]
 and the tautological section $s=\pi_* \sigma$. We claim that $\CM_{\overline{r},n}\cong Z(s) $.
 
 Let $ q=[\oO_{\BA^2}\onto Q]$ be a closed point of $ \Quot_{\BA^2}(\oO^r, n)$, and consider the diagram
   \[
\begin{tikzcd}
\CI^*\arrow[r, hook]& \oO_{\BA^2}^r\arrow[r, twoheadrightarrow]\arrow[dr] & \CE_{\overline{r}}\arrow[d, dashed, two heads]\\
& & Q.
\end{tikzcd}
\]
The point $q$ is in $\CM_{\overline{r},n}$ if and only if the dotted arrow $ \CE_{\overline{r}}\onto Q$ exists and makes the above diagram commutative. This last condition is equivalent to the vanishing $s(q)=0$. The same argument carries over flat families, by which we conclude the proof. 
\end{proof}
By the virtual pullback formula of \cite{KKP, Manolache-virtual-pb}, we obtain the compatibility of the virtual structures of $\CM_{\overline{r},n}$ and $\Quot_{\BA^2}(\oO^{r}, n) $.
\begin{corollary}\label{cor: v pull formula}
    We have an identification of $\TT$-equivariant virtual cycles
    \begin{align*}
        j_* [\CM_{\overline{r}, n}]^{\vir}&=e(\CI^{[n]})\cap [\Quot_{\BA^2}(\oO^{r}, n)]^{\vir}\in A^\TT_0(\Quot_{\BA^2}(\oO^{r}, n)),\\
        j_* \oO^{\vir}_{\CM_{\overline{r}, n}}&=\mathfrak{e}(\CI^{[n]})\otimes \oO^{\vir}_{\Quot_{\BA^2}(\oO^{r}, n)}\in K_0^\TT(\Quot_{\BA^2}(\oO^{r}, n)).
    \end{align*}
\end{corollary}
\begin{proof}
    Denote by $\BE_\CM, \BE_Q$  the (complex of the) obstruction theory respectively on $\CM=\CM_{\overline{r}, n}$ and $Q=\Quot_{\BA^2}(\oO^{r}, n) $. We claim that there exists a commutative diagram of $\TT$-equivariant exact triangles in $\derived^{b}_\TT(\CM_{\overline{r}, n})$
    \begin{equation}\label{eqn: v pullback}
        \begin{tikzcd}
j^*\BE_Q\arrow[r]\arrow[d]& \BE_\CM \arrow[d]\arrow[r]& (\CI^{[n]}|_\CM)^{*}[1]\arrow[d]\\
j^*\BL_{Q}\arrow[r]&\BL_{\CM} \arrow[r]& \BL_{j},
\end{tikzcd}
    \end{equation}
where $ (\CI^{[n]})^*$ is seen as a complex concentrated in degree 0 and $\BL_j$ is the relative cotangent complex. If the claim holds, then the required identities follow by the functoriality of virtual classes, see e.g.~\cite{KKP, Manolache-virtual-pb}.

To prove the claim, we use the description of $\CM, Q $ as moduli spaces of quiver representations inside $\mathcal M^{\nc}_{r,n}$.  Define the $\TT$-equivariant    vector bundle of rank $rn$
\begin{align}\label{eqn: taut W}
    \CW\defeq\bigoplus_{i=1}^2\Hom(W_i\cdot t_i, V)
\end{align}
on $\mathsf U_{r,n} $, 
which by $\GL(V)$-equivariance descends to a $\TT$-equivariant vector bundle on $\mathcal M^{\nc}_{r,n}$. Then there is an identity of $\TT$-equivariant vector bundles
\begin{align*}
    \CW|_{Q}=\CI^{[n]}
\end{align*}
on $Q$, which by the explicit description of $\BE_Q, \BE_\CM$  as in \eqref{eqn: obs th}
implies that 
\[
(\CI^{[n]}|_\CM)^{*}[1]\cong\Cone\left(j^*\BE_Q\to \BE_\CM   \right),
\]
and that in particular \eqref{eqn: v pullback} commutes.
\end{proof}
By \Cref{cor: v pull formula}, we can write the invariants of the gauge origami moduli space as ($K$-theoretic) $\TT$-equivariant intersection products on the Quot scheme $\Quot_{\BA^2}(\oO^{r}, n)$, by
\begin{align*}
      \chi\left(\mathcal{M}_{\overline{r}, n}, \oO^{\vir}_{\mathcal{M}_{\overline{r}, n}}\right)&=  \chi\left( \Quot_{\BA^2}(\oO^{r}, n), \oO^{\vir}_{\Quot_{\BA^2}(\oO^{r}, n)}\otimes \Lambda^\bullet(\CI^{[n]})^{*}\right),\\
     \int_{[\CM_{\overline{r}, n}]^{\vir}}1&=\int_{[\Quot_{\BA^2}(\oO^{r}, n)]^{\vir}}e(\CI^{[n]}).
\end{align*}
\begin{remark}
We remark that computing virtual invariants from the Quot scheme $\Quot_{\BA^2}(\oO^{r}, n) $ is general more complicated than doing it directly from $\CM_{\overline{r},n}$. In fact, performing virtual localization on $ \Quot_{\BA^2}(\oO^{r}, n)$ reduces these tautological  invariants to the harder combinatorics of Young diagrams, for which there are often no exact formulas. See for instance \cite{Lim_more_vir_sur, OP_quot_schemes_curves_surfaces, Stark_Quot_cos, Boj_Quot} for some related computations on $\Quot_{S}(\oO^{r}, n)$ in the case of a projective surface $S$ and \cite{BH_Quot} in the context of the Segre-Verlinde correspondence. 
It was noticed in  \cite{Lim_more_vir_sur} a compatibility of smooth and virtual cycles for Quot schemes defined on $C, S$, where $C$ is a \emph{smooth} canonical  curve of a smooth \emph{projective} surface $S$. In our case, we have that
\[\CC=Z(x_1x_2)\subset \BA^2\]
is a \emph{singular} $\TT_0$-equivariant canonical curve in $\BA^2$  with respect to the standard $\TT_0$-action on $\BA^2$, by being the zero locus of a $\TT_0$-equivariant section of the $\TT_0$-equivariant canonical bundle $\omega_{\BA^2}=\oO_{\BA^2}\cdot t_1t_2$. Therefore, we should see \Cref{cor: v pull formula} as a \emph{singular} and \emph{quasi-projective} analogue of the compatibility results of \cite{Lim_more_vir_sur}.
\end{remark}
\subsection{ADHM quiver}\label{sec: ADHM}
Let $r\geq 1$ and $n\geq 0$. A similar situation as in \Cref{sec: cut out in QUot} happens in relation to the moduli space $\CM^{\fr}_{r,n}$ of framed sheaves on $\BP^2$ along a fixed projective line $\BP^1\subset \BP^2$. Recall that $ \CM^{\fr}_{r,n}$ is defined as 
\begin{align*}
 \CM^{\fr}_{r,n}=\Set{
  (\mathscr E,\Phi)\,|
     \begin{array}{c}
       \mathscr E \in \Coh(\BP^2)\textrm{ is a }\textrm{ torsion free sheaf} \\
       \textrm{with }\ch(\mathscr E) = (r,0,-n), \textrm{ and }\Phi\colon \mathscr E|_{\BP^1} \simto \OO^{\oplus r}_{\BP^1}
     \end{array}
     } \Bigg{/}\sim  
\end{align*}
and is smooth irreducible of dimension $2nr$,  see e.g.~\cite{Nak_lectures_Hilb_schemes,CR_framed_sheaves}. The moduli space $\CM^{\fr}_{r,n}$ admits a well-known  description as Nakajima quiver variety, see \Cref{fig: AFHM}. 
\begin{figure}[ht]\contourlength{2.5pt}
    \centering
     \resizebox{0.3\linewidth}{!}{
    \begin{tikzpicture}
        \node (F123) at (0,0) {$\infty$};
        \node (Gk0) at (2.5,0) {$0$};
                \draw[<-] (Gk0) to[bend right=20] (F123); 
        \draw[->] (Gk0) to[bend left=20] (F123); 
        \draw[->] (Gk0) to[out=-100, in=-30, looseness=8] (Gk0);
        \draw[<-] (Gk0) to[out=30, in=100, looseness=8] (Gk0);
    \end{tikzpicture}
    }
    \caption{ADHM quiver.}  \label{fig: AFHM}
\end{figure}
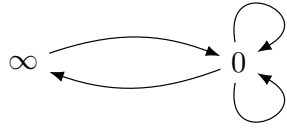

Set 
\[
\mathsf U^{\fr}_{r,n}=\Set{(B_1,B_2,I,J)\,|
\begin{array}{c}
[B_1,B_2]+IJ=0\textrm{, and there is no subspace} \\
S\subsetneq \BC^n\textrm{ such that }B_i(S)\subset S\textrm{ and }\mathrm{im}\,I \subset S
\end{array}
}, 
\]
where $(B_1,B_2,I,J)\in \End(V)^{\oplus 2}\oplus\Hom(W,V)\oplus\Hom(V,W)$ where $\dim W=r$ and $\dim V=n$. Then 
\[
\CM^{\fr}_{r,n}\cong \mathsf U^{\fr}_{r,n}/\GL(V).
\]
There is a natural $\TT$-action on $\CM^{\fr}_{r,n}$  given by
\begin{align}\label{eqn: torus ADHM}
    (t, w)\cdot (B_1, B_2,  I, J)=(t_1^{-1} B_1, t_2^{-1} B_2, \overrightarrow{w}^{-1}I, \overrightarrow{w}t_1^{-1}t_2^{-1}J),
\end{align}
 with notation as in \Cref{sec: ytorus action}.

 Set $W=W_1\oplus W_2$, with $\dim W_i=r_i$ and define the  vector bundles
 \begin{align*}
     \CF_i&=\Hom(W_i, V), \quad i=1,2,\\
     \CF&=\CF_1\oplus \CF_2,
 \end{align*}
 which by $\GL(V)$-equivariance descend to vector bundles on $\CM^{\fr}_{r,n}$, where $W_i$ are $\TT$-equivariantly decomposed as $W_i=\sum_{\alpha=1}^{r_i}\BC\otimes w_{i\alpha}$.

 It was already noticed in \cite[Sec.~3.2]{CR_framed_sheaves} that there is a closed embedding
 \[\tilde{j}:\Quot_{\BA^2}(\oO^{r}, n) \into  \CM^{\fr}_{r,n}\]
 realized as the zero locus $J=0$. Using the quiver description of both moduli spaces and the functoriality of virtual classes as in \Cref{cor: v pull formula},  we have an identification of $\TT$-equivariant virtual cycles\footnote{We twist  the sheaf $\CF^*$ by the character $t_1^{-1}t_2^{-1}$, which scales the required section $J$ compatibly with the torus action \eqref{eqn: torus ADHM}.}
    \begin{equation}\label{eqn: taut Fi}
    \begin{split}
        \tilde{j}_* [\Quot_{\BA^2}(\oO^{r}, n)]^{\vir}&=e(\CF^*\cdot t_1^{-1}t_2^{-1})\cap [ \CM^{\fr}_{r,n}]\in A^\TT_{rn}( \CM^{\fr}_{r,n}),\\
            \tilde{j}_* \oO^{\vir}_{\Quot_{\BA^2}(\oO^{r}, n)}&=\mathfrak{e}(\CF^*\cdot t_1^{-1}t_2^{-1})\in K_0^\TT( \CM^{\fr}_{r,n}).
            \end{split}
    \end{equation}
Notice that
\[
\bigoplus_{i=1}^2 \CF_i\cdot t_i^{-1}|_{\Quot_{\BA^2}(\oO^{r}, n)}=\CW, 
\]
with $\CW$ defined on \eqref{eqn: taut W}. Therefore,  combining \Cref{prop:zero sec Quot} with \eqref{eqn: taut Fi} and a manipulation as in \eqref{eqn: y deformation}, we obtain the following corollary.
\begin{corollary}\label{cor: ADHM vpull}
    Let $\overline{r}=(r_1, r_2)$ and $n\geq 0$. Let $ \iota: \CM_{\overline{r},n}\into  \CM^{\fr}_{r,n} $ denote the natural closed embedding. Then we have
      \begin{align*}
        \iota_* [ \CM_{\overline{r},n}]^{\vir} &=e\left(\CF_1\cdot t_1^{-1} \right)\cdot e\left(\CF_2\cdot t_2^{-1} \right)\cdot e(\CF^*\cdot t_1^{-1}t_2^{-1})\cap [ \CM^{\fr}_{r,n}]\in A^\TT_{0}( \CM^{\fr}_{r,n}),\\
            \iota_* \oO^{\vir}_{\CM_{\overline{r},n}} &=\Lambda_{-t_1}\CF_1^{*}\otimes \Lambda_{-t_2}\CF_2^{*}\otimes  \Lambda_{-t_1t_2}\CF\in K_0^\TT( \CM^{\fr}_{r,n}).
    \end{align*}
\end{corollary}
In particular, this implies that the gauge origami invariants defined via $\CM_{\overline{r}, n}$ can be recast as
\begin{align*}
 \chi\left(\mathcal{M}_{\overline{r}, n}, \oO^{\vir}_{\mathcal{M}_{\overline{r}, n}}\right)&=  \chi\left( \CM^{\fr}_{r,n},  \Lambda_{-t_1}\CF_1^{*}\otimes \Lambda_{-t_2}\CF_2^{*}\otimes  \Lambda_{-t_1t_2}\CF\right),\\
     \int_{[\CM_{\overline{r}, n}]^{\vir}}1&=\int_{\CM^{\fr}_{r,n}}e\left(\CF_1\cdot t_1^{-1} \right)\cdot e\left(\CF_2\cdot t_2^{-1} \right)\cdot e(\CF^*\cdot t_1^{-1}t_2^{-1}).
\end{align*}
The generating series of such  intersection numbers on $\CM^{\fr}_{r,n} $ are classically known as \emph{Nekrasov's partition functions} with \emph{fundamental matters} (resp~\emph{anti-fundamental matters}) given by the $K$-theoretic insertions $ \Lambda_{-t_1}\CF_1^{*},  \Lambda_{-t_2}\CF_2^{*}$ (resp. \emph{anti-fundamental matter} $\Lambda_{-t_1t_2}\CF$).
\subsection{Example: $\overline{r}=(0,1)$ and $n=1$} We conclude by working out explicitly the relations among the gauge origami moduli space, the Quot scheme and the moduli space of framed sheaves in the easiest case of $\overline{r}=(0,1)$ and $n=1$. We believe this guiding example will help the reader to go through the notation and the correct $\TT$-equivariant structure of the sheaves involved in the general construction.

The relevant moduli spaces (and their virtual dimensions) in this setting are
\begin{align*}
   & \mathcal{M}_{(0,1), 1}\cong \Hilb^1\BA^1\cong \BA^1, &\vdim( \mathcal{M}_{(0,1), 1})=0,\\
   &\Quot_{\BA^2}(\oO^{1}, 1)\cong \Hilb^1\BA^2\cong\BA^2,  &\vdim(\Quot_{\BA^2}(\oO^{1}, 1))=1,\\
   & \CM^{\nc}_{1,1}\cong \Hilb^1\BA^2\cong \BA^2,  &\dim(\CM^{\nc}_{1,1})=2,\\
   & \CM^{\fr}_{1,1}\cong \Hilb^1\BA^2\cong \BA^2. & \dim(\CM^{\fr}_{1,1})=2.
\end{align*}
While the last three moduli spaces are isomorphic, the virtual structure on the  Quot scheme does not coincide with its smooth structure, since its virtual dimension (i.e. the dimension of its virtual fundamental class) is smaller. Consider the chain of inclusions 
\begin{align*}
    p\in \mathcal{M}_{(0,1), 1} \hookrightarrow \Quot_{\BA^2}(\oO^{1}, 1) \hookrightarrow  \CM^{\nc}_{1,1} \cong   \CM^{\fr}_{1,1},
\end{align*}
where $p $ is the unique $\TT$-fixed point of all the four moduli spaces, corresponding to the quotient $[\iota_{2, *}\oO_{\BA^1}\onto \oO_0]$.

We summarise the (weight decomposition of the) fibers at the point $p$ of the various locally free sheaves used in constructing the virtual fundamental classes
\begin{align*}
    &\CV|_p= \Lambda^2\BC^2\otimes\End(V)\oplus\BC\otimes \Hom(W_2,V)=t_1^{-1}t_2^{-1}+t_2^{-1},\\
&\widetilde{\CV}|_p=\Lambda^2\BC^2\otimes\End(V)= t_1^{-1}t_2^{-1},\\
    &\CI^{[1]}|_p=\BC\otimes \Hom(W_2,V)= t_2^{-1},\\
    &\CF|_p^*\cdot t_1^{-1}t_2^{-1}=\Hom(W_2,V)\cdot t_1^{-1}t_2^{-1}=t_1^{-1}t_2^{-1}.
\end{align*}
The (virtual) tangent spaces are therefore
\begin{align*}
   T_{{\mathcal{M}_{(0,1), 1}}}^{\vir}|_p&= t_1^{-1}-t_1^{-1}t_2^{-1},\\
   T_{{\Quot_{\BA^2}(\oO^{1}, 1)}}^{\vir}|_p&= t_1^{-1}+t_2^{-1}-t_1^{-1}t_2^{-1},\\
     T_{\CM^{\fr}_{1,1}}|_p&= t_1^{-1}+t_2^{-1}.
\end{align*}
Applying the localisation formula in $K$-theory, we obtain
\begin{align*}
    \chi\left(\mathcal{M}_{(0,1), 1}, \oO^{\vir}_{\mathcal{M}_{(0,1), 1}}\right)&= \frac{1-t_1t_2}{1-t_1},\\
       \chi\left(\Quot_{\BA^2}(\oO^{1}, 1), \oO^{\vir}_{\Quot_{\BA^2}(\oO^{1}, 1)}\right)&=\frac{1-t_1t_2}{(1-t_1)(1-t_2)},\\
        \chi\left(\mathcal{M}^{\fr}_{1, 1}, \oO_{\mathcal{M}^{\fr}_{1, 1}}\right)&= \frac{1}{(1-t_1)(1-t_2)},
\end{align*}
in complete agreement with the pushforward formulas of \eqref{eqn: taut Fi}, \Cref{cor: v pull formula} and \Cref{cor: ADHM vpull}.
\subsection*{Conflict of interest} None.
\subsection*{Financial Support} S.M. was supported by the FNS Project 200021-196960 ``Arithmetic aspects of moduli spaces on
curves'', and by  the HORIZON-MSCA-2024-PF-01 Project 
101203281 ``Moduli Spaces of Sheaves: Geometry and Invariants'', funded by the Research and Innovation framework programme Horizon Europe {\normalsize\euflag}.  S.M. is a  member of GNSAGA of INDAM, Italy.
\bibliographystyle{amsplain-nodash}
\bibliography{The_Bible}
\end{document}